\documentclass[11pt,a4paper]{amsart}

\usepackage{amsfonts,amsthm,amssymb,latexsym,amsmath}
\usepackage{setspace,amscd,verbatim,amsrefs}
\usepackage[pdftex]{graphicx}
\usepackage{graphics}
\usepackage[matrix,arrow]{xy}
\usepackage[active]{srcltx}
\usepackage[colorlinks=true]{hyperref}
\usepackage{mathrsfs}
\usepackage{enumitem}
\usepackage{color}

\theoremstyle{plain}
\newtheorem{lemma}{Lemma}[section]
\newtheorem{theorem}[lemma]{Theorem}
\newtheorem{corollary}[lemma]{Corollary}

\theoremstyle{definition}
\newtheorem{definition}[lemma]{Definition}
\newtheorem{remark}[lemma]{Remark}

\theoremstyle{remark}

\newcommand{\R}{\mathbb{R}}
\newcommand{\C}{\mathbb{C}}
\newcommand{\Z}{\mathbb{Z}}
\newcommand{\D}{\mathbb{D}}

\newcommand{\link}{\operatorname{link}}
\renewcommand{\sl}{\operatorname{sl}}
\newcommand{\wind}{\operatorname{wind}}
\newcommand{\CZ}{\operatorname{CZ}}
\newcommand{\ECH}{\operatorname{ECH}}
\newcommand{\ECC}{\operatorname{ECC}}

\renewcommand{\P}{\mathcal{P}}
\newcommand{\action}{\mathcal{A}}

\newcommand{\jbar}{\bar{J}}
\newcommand{\jtil}{\widetilde{J}}
\newcommand{\util}{\widetilde{u}}
\newcommand{\vtil}{\widetilde{v}}
\newcommand{\wtil}{\widetilde{w}}

\newcommand{\Mfast}{\mathcal{M}^{\rm fast}}

\newcommand{\W}{\overline{W}}

\numberwithin{equation}{section}

\begin{document}

\title[]{Hopf orbits and the first ECH capacity}
\author[]{Umberto Hryniewicz, Michael Hutchings, and Vinicius G. B. Ramos}
\date{\today}

\begin{abstract}
We consider dynamically convex star-shaped domains in a symplectic vector space of dimension~$4$. For such a domain, a ``Hopf orbit'' is a closed characteristic in the boundary which is unknotted and has self-linking number~$-1$. We show that the minimum action among Hopf orbits exists and defines a symplectic capacity for dynamically convex star-shaped domains. We further show that this capacity agrees with the first ECH capacity for such domains. Combined with a result of Edtmair~\cite{gss_edtmair}, this implies that for dynamically convex star-shaped domains in four dimensions, the first ECH capacity agrees with the cylinder capacity. 
This also provides a method to show that the first ECH capacity of a dynamically convex star-shaped domain satisfies the axioms of a normalized symplectic capacity without any need for Seiberg-Witten theory.
\end{abstract}


\maketitle

\tableofcontents

\section{Introduction}

The theme of this paper is the study of symplectic capacities of dynamically convex star-shaped domains in a $4$-dimensional symplectic vector space. In order to state our results and related questions we first need to recall a few basic notions.

\subsection{Symplectic capacities}

Recall that if $(X_1,\omega_1)$ and $(X_2,\omega_2)$ are symplectic manifolds of the same dimension, a {\bf symplectic embedding\/} of $X_1$ into $X_2$ is a smooth embedding $\varphi:X_1\hookrightarrow X_2$  such that $\varphi^*\omega_2=\omega_1$. In this paper, the symplectic manifolds that we consider will usually be domains\footnote{Here we define a ``domain'' in $\R^{2n}$ to be the closure of a nonempty open set.} in $\R^{2n} = \C^n$ with smooth boundary, usually with $n=2$, which we always equip with the restriction of the standard symplectic form
\[
\omega_0 = \sum_{j=1}^n dx_j \wedge dy_j.
\]
Here $(x_1,\dots,x_n,y_1,\dots,y_n)$ denote the coordinates in $\R^{2n}$. The corresponding coordinates in $\C^n$ are denoted by $(z_1,\dots,z_n)$, where $z_j = x_j + iy_j \in \C$. In particular, define the ball
\[
B^{2n}(r)=\{z\in\C^n\mid \pi|z|^2\le r\}
\]
and the cylinder
\[
Z^{2n}(R)=\{z\in\C^n\mid \pi|z_1|^2\le R\}.
\]
In 1985 Gromov proved the celebrated nonsqueezing theorem~\cite{gromov}, asserting that there exists a symplectic embedding $B^{2n}(r)\hookrightarrow Z^{2n}(R)$ if and only if $r\le R$. Since then, symplectic embeddings have been extensively studied, see e.g.\ \cite{schlenk}.

One of the main tools in this study are symplectic capacities, see e.g.\ the survey~\cite{chls}. Definitions of this term in the literature vary, and we will use the following:

\begin{definition}
A {\bf symplectic capacity\/} defined in dimension $2n$ is a function~$c$ from some nonempty set of $2n$-dimensional symplectic manifolds to $[0,\infty]$, such that:
\begin{itemize}
\item (Monotonicity) If $(X_1,\omega_1)$ and $(X_2,\omega_2)$ are symplectic manifolds of dimension $2n$ for which $c$ is defined, and if there exists a symplectic embedding $X_1\hookrightarrow X_2$, then
\begin{equation}
\label{eqn:monotonicity}
c(X_1)\le c(X_2).
\end{equation}
\item (Conformality) If $r>0$ and if $c$ is defined on $(X,\omega)$, then
\begin{equation}
\label{eqn:conformality}
c(X,r\omega)=r c(X,\omega).
\end{equation}
In particular, if $X$ is a domain in $\R^{2n}$ then $c(rX)=r^2c(X)$ for all $r \in \mathbb{R} \setminus \{0\}$.
\end{itemize}
We say that a symplectic capacity $c$ is {\bf normalized\/} if in addition
\[
c(B^{2n}(r)) = c(Z^{2n}(r)) = r.
\]
\end{definition}

The simplest examples of normalized symplectic capacities are the \textbf{Gromov width}~$c_B$, sometimes also referred to as the \textbf{ball capacity}, and the {\bf cylinder capacity\/}~$c_Z$. If $\operatorname{dim}(X)=2n$, then these are defined by
\begin{eqnarray}
\nonumber
    c_{B}(X,\omega) &= \sup\{r \mid \mbox{$\exists$ symplectic embedding $B^{2n}(r)\hookrightarrow (X,\omega)$}\},\\
\label{eqn:cZ}
    c_Z(X,\omega) &= \inf\{R \mid \mbox{$\exists$ symplectic embedding $(X,\omega)\hookrightarrow Z^{2n}(R)$}\}.
\end{eqnarray}
The statements that $c_B$ and $c_Z$ are normalized symplectic capacities are equivalent to Gromov's nonsqueezing theorem~\cite{gromov}. In addition, it follows from the definitions that if $c$ is a normalized symplectic capacity defined on $(X,\omega)$, then
\begin{equation}
\label{eq:cap}
c_B(X,\omega)\le c(X,\omega)\le c_Z(X,\omega).
\end{equation}

Various symplectic capacities are defined in terms of Reeb dynamics and contact homology. To say more about this, we need the following definitions.

\begin{definition}
A {\bf star-shaped domain\/} is a compact domain $X\subset\R^{2n}$ with smooth connected boundary such that $0\in\operatorname{int}(X)$ and $\partial X$ is transverse to the radial vector field.
\end{definition}

If $X$ is a star-shaped domain as above, then the standard Liouville form on $\R^{2n}$ defined by
\[
\lambda_0 = \frac{1}{2}\sum_{i=1}^n\left(x_i\,dy_i - y_i\,dx_i\right)
\]
restricts to a contact form $\lambda$ on $\partial X$. The contact form $\lambda$ determines the {\bf Reeb vector field\/} $R_\lambda$ characterized by $d\lambda(R_\lambda,\cdot)=0$ and $\lambda(R_\lambda)=1$. A {\bf Reeb orbit\/} is a periodic orbit of $R_\lambda$, i.e.\ a map $\gamma:\R/T\Z\to\partial X$ for some $T>0$, modulo 
translations of $\R/T\Z$, such that $\dot\gamma(t)=R_\lambda(\gamma(t))$ for all $t$. The Reeb orbit $\gamma$ is {\bf simple\/} if the map $\gamma$ is an embedding, i.e. $T>0$ is the smallest positive period. The {\bf symplectic action\/} $\mathcal{A}(\gamma)>0$ is defined to be the period $T = \int_\gamma \lambda$. By a theorem of Rabinowitz \cite{rabinowitz}, at least one Reeb orbit exists. Let $\mathcal{P}(\partial X)$ denote the set of Reeb orbits in $\partial X$, and define
\[
\mathcal{A}_{\operatorname{\min}}(X) = \min_{\gamma\in\mathcal{P}(\partial X)}\mathcal{A}(\gamma).
\]
The \textbf{action spectrum} of the star-shaped domain~$X$ is defined to be the set of periods of Reeb orbits in~$\partial X$. A symplectic capacity~$c$ is said to be \textbf{spectral} on star-shaped domains if for every star-shaped domain $X$ we have \begin{equation}
\label{eqn:spectralcapacity}
c(X) = \sum_i\mathcal{A}(\gamma_i)
\end{equation}
for some finite collection $\{\gamma_i\}$ of Reeb orbits in $\partial X$.


The ECH capacities are a sequence of symplectic capacities $\{c_k^{\operatorname{ECH}}\}_{\ge 0}$ defined in four dimensions \cite{hut_qech}. If $(X,\omega)$ is a four-dimensional symplectic manifold, then
\[
0 = c_0^{\operatorname{ECH}}(X,\omega) < c_1^{\operatorname{ECH}}(X,\omega) \le c_2^{\operatorname{ECH}}(X,\omega) \le \cdots \le \infty.
\]
For each $k$ the capacity $c_k^{\rm ECH}$ is spectral on star-shaped domains, i.e. if $X\subset \R^4$ is a star-shaped domain then~\eqref{eqn:spectralcapacity} holds, where
the  finite set of Reeb orbits $\{\gamma_i\}$ on $\partial X$ is selected homologically by embedded contact homology. The first ECH capacity $c_1^{\operatorname{ECH}}$ is normalized, and we will review its definition for star-shaped domains in~\S\ref{sec:ckech}.

There are also ``alternative ECH capacities'' $c_k^{\operatorname{Alt}}$ defined by a more elementary max-min formula in~\cite{altech}, which satisfy the same basic properties as those of the ECH capacities listed above and agree with them on basic examples. If $X$ is a star-shaped domain, then by \cite{altech}*{Thm.~12} we have
\[
c_k^{\operatorname{Alt}}(X) \le c_k^{\operatorname{ECH}}(X).
\]

Some additional sequences of spectral symplectic capacities, defined for star-shaped domains in $\R^{2n}$, are the Ekeland-Hofer capacities $\{c_k^{\operatorname{EH}}\}_{k\ge 1}$ defined in \cite{ekelandhofer} using variational methods. Recent work in~\cite{guttramos} establishes that these capacities are equal to the capacities $\{c_k^{\operatorname{CH}}\}_{k\ge 1}$ defined in \cite{gutthutchings} using equivariant symplectic homology. The capacity $c_1^{\operatorname{EH}} = c_1^{\operatorname{CH}}$ is normalized. 

\subsection{The Hopf capacity}

Let $X\subset\R^4$ be a star-shaped domain and let $\lambda$ denote the contact form on $\partial X$. The contact structure $\xi = \operatorname{Ker}(\lambda)$ has, up to homotopy, a unique symplectic trivialization $\tau$ over $\partial X$. If $\gamma$ is a Reeb orbit, then the linearized Reeb flow along $\gamma$ has a well-defined rotation number with respect to $\tau$, which we denote by $\rho(\gamma)\in\R$; see \cite[\S1.2]{Hutchings_Ruelle} for details.

\begin{definition}
A star-shaped domain $X\subset\R^4$ is {\bf dynamically convex\/} if every Reeb orbit $\gamma$ in $\partial X$ has rotation number $\rho(\gamma)>1$.
\end{definition}

\begin{remark}
The definition of dynamical convexity is more commonly stated as the condition that every Reeb orbit $\gamma$ in $\partial X$ has Conley-Zehnder index $\operatorname{CZ}(\gamma)\ge 3$. Here if $\gamma$ is a nondegenerate Reeb orbit (see \S\ref{sec:preliminaries} for the definition of ``nondegenerate''), then
\begin{equation}
\label{eqn:CZ}
\operatorname{CZ}(\gamma)=\lfloor\rho(\gamma)\rfloor + \lceil\rho(\gamma)\rceil.
\end{equation}
If $\gamma$ is degenerate, which in particular implies that $\rho(\gamma)$ is an integer, then there are different possible conventions for the Conley-Zehnder index. We will use the maximal lower semicontinuous extension of the nondegenerate Conley-Zehnder index, which in the degenerate case sometimes agrees with \eqref{eqn:CZ} and sometimes is $1$ less; see \S\ref{sec:CZ} for the precise definition. With this convention, $\rho(\gamma)>1$ is equivalent to $\operatorname{CZ}(\gamma)\ge 3$.
\end{remark}

The notion of dynamical convexity was introduced by Hofer, Wysocki and Zehnder in~\cite{convex}, as a symplectomorphism-invariant alternative to the notion of convexity. It is shown in \cite{convex} that every convex star-shaped domain in $\R^4$ whose boundary is positively curved is dynamically convex. For a long time it was not known whether, conversely, every dynamically convex domain is symplectomorphic to a convex domain. It was recently shown by Chaidez and Edtmair in~\cite{chaidezedtmair} that the answer is no; see \cite{dgrz,c-gh} for further developments. 


Recall that the {\bf self-linking number\/} of a simple Reeb orbit in $\partial X$, or more generally any knot transverse to the contact structure $\xi$, is its linking number with a pushoff via the framing induced by the global trivialization $\tau$ of $\xi$.

\begin{definition}
A {\bf Hopf orbit\/} is a simple Reeb orbit $\gamma$ in $\partial X$ which is unknotted and has self-linking number $-1$. Let $\mathcal{P}_{\operatorname{Hopf}}(\partial X)$ denote the set of Hopf orbits in~$\partial X$.
\end{definition}

\begin{remark}
\label{rem:gss}
It is shown in \cite{convex} that if $X$ is dynamically convex,  then the Reeb flow on $\partial X$ admits a {\bf disk-like global surface of section\/}. This is an embedded disk $D\subset\partial X$ such that $\partial D$ is a Reeb orbit $\gamma$, the Reeb flow is transverse to $\mathrm{int}(D)$, and for every $x\in\partial X \setminus \gamma$, the Reeb trajectory through $x$ intersects $\mathrm{int}(D)$ forwards and backwards in time. It follows from these conditions that $\gamma$ is a Hopf orbit. Conversely, it is shown in \cite{hryn_jsg} that if $X$ is dynamically convex, then every Hopf orbit in $\partial X$ bounds a disk-like global surface of section for the Reeb flow.
\end{remark}

\begin{definition}
If $X\subset\R^4$ is a dynamically convex domain, define
\[
\action_{\rm Hopf}(X) = \inf \ \{ \action(\gamma) \mid \gamma \in \P_{\rm Hopf}(\partial X) \}.
\]
\end{definition}

We can use $\action_{\rm Hopf}$ in a standard way to define an invariant of symplectic $4$-manifolds as follows.

\begin{definition}
If $(X,\omega)$ is a symplectic four-manifold, define
\begin{equation}
\label{c_Hopf}
c_{\rm Hopf}(X,\omega) = \sup \{ \mathcal{A}_{\rm Hopf}(\Omega)\mid \Omega \text{ is dynamically convex, }(\Omega,\omega_0) \hookrightarrow(X,\omega)\}
\end{equation}
where the supremum is taken over all dynamically convex star-shaped domains $\Omega \subset \R^4$ such that $(\Omega,\omega_0)$ symplectically embeds into $(X,\omega)$.
\end{definition}

Our first result is the following.

\begin{theorem}
\label{main_fast_capacity}
The invariant $c_{\rm Hopf}$ is a normalized symplectic capacity. Moreover, if $X\subset \R^4$ is a dynamically convex domain, then
\[
c_{\rm Hopf}(X) = \action_{\rm Hopf}(X) = \action(\gamma)
\]
for some $\gamma \in \P_{\rm Hopf}(X)$ satisfying $\CZ(\gamma)=3$.
\end{theorem}

\begin{remark}
\label{rem:HWZquestionGSS}
Theorem~\ref{main_fast_capacity} and Remark~\ref{rem:gss} relate to the following important question posed by Hofer, Wysocki and Zehnder in~\cite{convex}: \textit{For a dynamically convex star-shaped domain $X$ in $\R^4$, is $\mathcal{A}_{\rm min}(X)$ equal to the action of a Reeb orbit that bounds a disk-like global surface of section for the Reeb flow on~$\partial X$?} In view of the result from~\cite{hryn_jsg}, this is equivalent to asking: \textit{Is $\action_{\rm min}(X) = \action_{\rm Hopf}(X)$ for every dynamically convex star-shaped domain $X$ in $\R^4$?} This difficult question remains open. Abbondandolo, Edtmair and Kang showed in~\cite{aek} that it has an affirmative answer in the special case where $X$ is convex.
\end{remark}

\begin{remark}
\label{rem:fastplanes}
The key technical step in the proof of Theorem~\ref{main_fast_capacity} is Lemma~\ref{lemma_hard_step} below. Its proof is based on the notion of a ``fast plane'' introduced in~\cite{hry_thesis}. An end of a finite-energy curve can be seen as a gradient trajectory of the action functional on the loop space; it converges to a periodic orbit, i.e. a critical point of the action. There is an associated asymptotic operator, which plays the role of the ``Hessian'' of the action at the critical point; in a nondegenerate situation it governs the approach to the periodic orbit, in the very same way that in finite-dimensions the Hessian governs how a gradient trajectory approaches a critical point. This is the main result of the work~\cite{props1} by Hofer, Wysocki and Zehnder. Still assuming nondegeneracy, a curve approaches the periodic orbit at a given puncture with an exponential decay dictated by a negative eigenvalue. The more negative the eigenvalue, the faster the approach. In $3$-dimensions, eigenvalues can be compared to each other in terms of winding numbers of associated eigenvectors: smaller winding number implies smaller eigenvalue. Moreover, for topological reasons, the winding number of the ``asymptotic eigenvector'' of a finite-energy plane is~$\geq 1$ in a trivialization of the contact structure that extends over the plane. A fast plane is one for which the approach is as fast as possible, in the sense that the corresponding eigenvalue has eigenvectors of winding number~$1$. It turns out that fast planes have good intersection-theoretic properties to form foliations transverse to the Reeb flow and, under the assumption of dynamical convexity, to induce global surfaces of section.
\end{remark}


\subsection{The first ECH capacity}

Our second result is the following dynamical interpretation of the first ECH capacity.

\begin{theorem}
\label{thm:Hopf_ECH}
If $X\subset\R^4$ is a dynamically convex domain, then
\[
c_1^{\rm ECH}(X) = \mathcal{A}_{\operatorname{Hopf}}(X).
\]
\end{theorem}

\begin{remark}
It is shown in~\cite{hut_qech} using Seiberg-Witten theory that $c_1^{\rm ECH}$ is a normalized symplectic capacity. It follows from this and Theorem 1.12 that $c_{\rm Hopf}$ is a normalized symplectic capacity on dynamically convex star-shaped domains in~$\R^4$. However our proofs of Theorems 1.9 and Theorem 1.12 establish that $c_1^{\rm ECH}$ and $c_{\rm Hopf}$ are normalized symplectic capacities on dynamically convex star-shaped domains in $\R^4$ without using Seiberg-Witten theory.
\end{remark}

\begin{remark}
The first ECH capacity of a dynamically convex domain $X\subset\R^4$ is the first number in the ECH spectrum of $\partial X$ with the induced contact form. Shibata \cite{taisuke1,taisuke2} has independently established dynamical interpretations of the first number in the ECH spectrum, similar to Theorem~\ref{thm:Hopf_ECH} and with some similar arguments, for certain lens spaces with contact forms that are either strictly convex or dynamically convex and nondegenerate. Note that Theorem~\ref{thm:Hopf_ECH} does not make any nondegeneracy assumption.
\end{remark}

\subsection{Relation to the cylinder capacity}

Let $X\subset\R^4$ be a dynamically convex domain. It follows from Theorem \ref{main_fast_capacity} or Theorem~\ref{thm:Hopf_ECH} and \eqref{eq:cap} that
\begin{equation}
\label{eqn:Hopf_cyl_ineq}
\action_{\rm Hopf}(X)\le c_Z(X).
\end{equation}
On the other hand, Edtmair~\cite{gss_edtmair} showed that if there exists a Hopf orbit $\gamma$ with $\mathcal{A}(\gamma)=R$, then there exists a symplectic embedding $X\hookrightarrow Z(R)$. In particular, $c_Z(X) \le \mathcal{A}_{\operatorname{Hopf}}(X)$. Combining this with Theorem~\ref{main_fast_capacity}, or alternatively with Theorem~\ref{thm:Hopf_ECH}, we get:

\begin{corollary}
\label{cor:chopf_cz}
If $X\subset\R^4$ is a dynamically convex domain, then
\begin{equation}
\label{eqn:Hopf_cyl}
c_Z(X) = \action_{\rm Hopf}(X),
\end{equation}
and the infimum in \eqref{eqn:cZ} is realized.
\end{corollary}

\begin{remark}
Corollary~\ref{cor:chopf_cz} is also stated in~\cite{gss_edtmair}. However, it should be stressed that the work done in~\cite{gss_edtmair} concerning~\eqref{eqn:Hopf_cyl} consists solely in the proof of the inequality $c_Z(X) \le \mathcal{A}_{\operatorname{Hopf}}(X)$ for a dynamically convex star-shaped domain~$X$. The reverse inequality~$\mathcal{A}_{\operatorname{Hopf}}(X) \leq c_Z(X)$ is a direct consequence of Theorem~\ref{main_fast_capacity}.
This is used in~\cite{gss_edtmair}. 
\end{remark}



\subsection{Viterbo's conjecture and the convex case}


Viterbo's conjecture from~\cite{viterbo} asserts that if $c$ is a normalized symplectic capacity and $X\subset\R^{2n}$ is convex then $c(X)^n \le n! \, {\rm vol}(X)$. This would follow from a positive answer  to a question due to Hofer from the 1990s: \textit{Do the values of all normalized symplectic capacities agree on a convex domain?} A conjectural positive answer to this question has been later called the ``strong Viterbo conjecture''. Its validity implies the ``weak Viterbo conjecture'' asserting that $\action_{\rm min}(X)^n \le n! \, {\rm vol}(X)$ when $X$ is a convex domain with smooth boundary. See~\cite{GHR} for discussion. Recently Haim-Kislev and Ostrover ~\cite{ako} found counterexamples to the weak Viterbo conjecture, and hence to all versions of the conjecture, for all $n \ge 2$. On the other hand, the result of~\cite{aek} shows that if $X\subset \R^4$ is convex, then all normalized capacities of $X$ that are spectral agree with $\action_{\rm min}(X)$. It is an open question whether this is still true if $X \subset \R^4$ is dynamically convex. For example, if $X\subset\R^4$ is dynamically convex, then we know that $c_1^{\rm Alt}(X) \le c_1^{\rm CH}(X) \le c_1^{\rm ECH}(X) = c_Z(X)$, but we do not know whether equality always holds. This would follow from the stronger conjecture in Remark~\ref{rem:HWZquestionGSS}.

\subsection*{Acknowledgments.} U.H.\ acknowledges the support by DFG SFB/TRR 191 `Symplectic Structures in Geometry, Algebra and Dynamics', Projektnummer 281071066-TRR 191. M.H.\ was partially supported by NSF grant DMS-2005437. V.G.B.R.\ is grateful for the hospitality of the Institute for Advanced Study, where part of this work was completed. V.G.B.R.\ was partially supported by NSF grant DMS-1926686, FAPERJ grant JCNE E-26/201.399/2021 and a Serrapilheira Institute grant.


\section{The Hopf capacity}
\label{sec:preliminaries}

The goal of this section is to prove Theorem \ref{main_fast_capacity}. Before stating the main lemma, we recall a few definitions from contact dynamics.

Let $X\subset \R^4$ be a star-shaped domain, and let $\lambda=\lambda_0|_{\partial X}$ be the induced contact form on $\partial X$. Then $d\lambda$ restricts to a symplectic structure on the contact structure $\xi=\operatorname{Ker}(\lambda)$. We denote by $\phi^t:\partial X \to \partial X$ the flow of the Reeb vector field $R_\lambda$. The derivative of the flow restricts to a symplectic linear map
\[
d\phi^t:(\xi_p,d\lambda)\longrightarrow (\xi_{\phi^t(p)},d\lambda).
\]

Now let $\gamma:\R/T\Z\to\partial X$ be a Reeb orbit.
The Reeb orbit $\gamma$ is {\bf nondegenerate\/} if $1$ is not in the spectrum of 
\begin{equation}
\label{eqn:lrm}
d\phi^T : {\xi}_{\gamma(0)} \longrightarrow {\xi}_{\gamma(0)}.
\end{equation}
We say that the domain $X$ (or the contact form on $\partial X$) is {\bf nondegenerate\/} (resp.\ nondegenerate up to action $C$) if all Reeb orbits $\gamma\in\P(\partial X)$ (resp.\ all Reeb orbits with action less than $C$) are nondegenerate. Finally, we say that the domain $X$ (or the contact form on $\partial X$) is {\bf dynamically convex up to action $C$\/} if $\CZ(\gamma)\ge 3$ for all $\gamma\in \P(\partial X)$ with $\action(\gamma)\leq C$.


\subsection{The main lemma and the proof of Theorem~\ref{main_fast_capacity}}

The crucial step in the proof of Theorem~\ref{main_fast_capacity} is the following result.

\begin{lemma}
\label{lemma_hard_step}
Let $\Omega_1$ and $\Omega_2$ be star-shaped domains in $\R^4$.
Let $C>\action_{\rm Hopf}(\Omega_2)$ and assume that:
\begin{itemize}
\item $\Omega_1$ is nondegenerate up to action~$C$.
\item $\Omega_1$, $\Omega_2$ are dynamically convex up to action $C$.
\end{itemize}
Suppose that there exists a symplectic embedding $\varphi: \Omega_1 \to \Omega_2$. Then there exists a Hopf orbit $\gamma\in\P_{\rm Hopf}(\Omega_1)$ such that $\CZ(\gamma)=3$ and $\action(\gamma) \leq \action_{\rm Hopf}(\Omega_2)$. In particular, $\mathcal{A}_{\rm Hopf}(\Omega_1) \leq \mathcal{A}_{\rm Hopf}(\Omega_2)$. If, in addition, $\Omega_2$ is nondegenerate up to action~$C$ and $\varphi(\Omega_1) \subset {\rm int}(\Omega_2)$ then $\action(\gamma) < \action_{\rm Hopf}(\Omega_2)$.
\end{lemma}

\begin{proof}[Proof of Theorem~\ref{main_fast_capacity}, assuming Lemma~\ref{lemma_hard_step}.]

We proceed in four steps.

\subsubsection*{Step 1}
We first show that $c_{\operatorname{Hopf}}$ is a symplectic capacity.

It follows immediately from equation \eqref{c_Hopf} that $c_{\operatorname{Hopf}}$ satisfies the Monotonicity property \eqref{eqn:monotonicity}.

We next show that $c_{\operatorname{Hopf}}$ satisfies the Conformality property \eqref{eqn:conformality}. By equation \eqref{c_Hopf}, it is enough to show that if $\Omega\subset\R^4$ is a dynamically convex domain and $r>0$, then
\begin{equation}
\label{eqn:conformality2}
\mathcal{A}_{\operatorname{Hopf}}(r\Omega) = r^2\mathcal{A}_{\operatorname{Hopf}}(\Omega).
\end{equation}
Equation \eqref{eqn:conformality2} holds because the scaling map $\Omega\to r\Omega$ preserves the direction of the Reeb vector field but multiplies the period of all Reeb orbits by $r^2$.

%

\subsubsection*{Step 2}
We now show that if $X\subset\R^4$ is any dynamically convex domain then
\begin{equation}
\label{eq:chopf_ahopf}
    c_{\rm Hopf}(X)=\action_{\rm Hopf}(X).
\end{equation}
To prove \eqref{eq:chopf_ahopf}, first we need an auxiliary lemma.

\begin{lemma}
\label{lemma_easy_step}
Let $\Omega \subset \mathbb{R}^4$ be a star-shaped domain that is dynamically convex up to~$C>\action_{\rm Hopf}(\Omega)$. Let $\Omega_n \subset \mathbb{R}^4$ be star-shaped domains satisfying $\partial \Omega_n \to \partial \Omega$ in $C^\infty$ as hypersurfaces. Let $\gamma_n \in \mathcal{P}_{\rm Hopf}(\partial\Omega_n)$ satisfy $\sup_n \mathcal{A}(\gamma_n) \leq C$ and ${\rm CZ}(\gamma_n)=3 \ \forall n$. There exists $\gamma \in \mathcal{P}_{\rm Hopf}(\partial \Omega)$ satisfying $\mathcal{A}(\gamma) = \liminf_{n \to \infty} \mathcal{A}(\gamma_n)$ and ${\rm CZ}(\gamma)=3$.
\end{lemma}

It is important to note that the above statement makes no nondegeneracy assumptions.

\begin{proof}[Proof of Lemma~\ref{lemma_easy_step}]
The statement can be rewritten in terms of contact forms on~$S^3$ defining its standard contact structure $\xi_0 = \ker \lambda_0|_{S^3}$. Let $\lambda$, $\lambda_1,\lambda_2,\dots$ be contact forms on $(S^3,\xi_0)$. Assume that $\lambda$ is dynamically convex up to action~$C$ and that $\lambda_n \to \lambda$ in $C^\infty$. Let $\gamma_n$ be periodic $\lambda_n$-Reeb orbits which are Hopf orbits, i.e. each $\gamma_n$ is simple, unknotted and has self-linking number~$-1$. Assume that ${\rm CZ}(\gamma_n)=3 \ \forall n$ and that $\sup_n \int_{\gamma_n}\lambda_n \leq C$. We need to show that there exists a simple periodic $\lambda$-Reeb orbit~$\gamma$ which is unknotted, has self-linking number~$-1$, and $\int_\gamma\lambda = \liminf_{n\to\infty} \int_{\gamma_n}\lambda_n$. In fact, since $\int_{\gamma_n} \lambda_n$ is the period of $\gamma_n$ as a $\lambda_n$-Reeb orbit, we can apply the Arzel\`a-Ascoli Theorem  and assume, without loss of generality and up to selection of a subsequence, that for some periodic $\lambda$-Reeb orbit~$\gamma$ we have $\gamma_n \to \gamma$ in~$C^\infty$. The~$\gamma_n$ are simple by assumption but, in principle, it might not be the case that~$\gamma$ is simple. If $\gamma$ is not simple then, by the dynamical convexity of $\lambda$ up to action~$C$ and the fact that $\int_\gamma\lambda \leq C$, we have ${\rm CZ}(\gamma) \geq 5$. However, the lower-semicontinuity of the Conley-Zehnder index and the dynamical convexity of~$\Omega$ up to action~$C$ together imply that $3 \leq {\rm CZ}(\gamma) \leq \liminf_{n\to\infty} {\rm CZ}(\gamma_n) = 3$. This shows that $\gamma$ is a simple periodic orbit with ${\rm CZ}(\gamma)=3$, and that $\gamma$ is transversely isotopic to~$\gamma_n$ for~$n$ large enough. Hence, $\gamma$ is unknotted with self-linking number~$-1$.
\end{proof}

Suppose that $\Omega$ is dynamically convex and that there exists a symplectic embedding $\Omega\hookrightarrow X$. We will now show that
\begin{equation}
\label{eqn:wnts}
\mathcal{A}_{\operatorname{Hopf}}(\Omega) \le \mathcal{A}_{\operatorname{Hopf}}(X).
\end{equation}
Let $\lambda$ denote the contact form on $\partial\Omega$. Consider $\Omega_n$ nondegenerate star-shaped domains satisfying $\Omega_n \subset {\rm int}(\Omega)$, and $\Omega_n \to \Omega$ in $C^\infty$. We then have symplectic embeddings~$\Omega_n \hookrightarrow \mathrm{int}(X)$. For every $C>0$ there exists $n_C \in \mathbb{N}$ such that if $n\geq n_C$ then $\Omega_n$ is dynamically convex up to action~$C$. This is a simple consequence of the Arzel\`a-Ascoli Theorem and the fact that the Conley-Zehnder index is lower-semicontinuous. By Lemma~\ref{lemma_hard_step}, we find $\gamma_n \in \mathcal{P}_{\rm Hopf}(\partial \Omega_n)$ with ${\rm CZ}(\gamma_n) = 3$ and $\mathcal{A}(\gamma_n) \leq \mathcal{A}_{\rm Hopf}(X)$. By an application of Lemma~\ref{lemma_easy_step}, there exists $\gamma \in \mathcal{P}_{\rm Hopf}(\partial \Omega)$ such that $\mathcal{A}(\gamma) \leq \mathcal{A}_{\rm Hopf}(X)$. In particular,~\eqref{eqn:wnts} holds.

The inequality $c_{\rm Hopf}(X) \leq \mathcal{A}_{\rm Hopf}(X)$ is a direct consequence of the fact that~\eqref{eqn:wnts} holds for every dynamically convex star-shaped domain~$\Omega$ that symplectically embeds into $X$. The inequality $\mathcal{A}_{\rm Hopf}(X) \leq c_{\rm Hopf}(X)$ follows from the definition of~$c_{\rm Hopf}$. The proof of Step~2 is complete.

\subsubsection*{Step 3}
We now show that the symplectic capacity $c_{\operatorname{Hopf}}$ is normalized. Given real numbers $a,b>0$, define the ellipsoid
\[
E(a,b) = \left\{z\in\C^2 \;\bigg|\; \frac{\pi|z_1|^2}{a} + \frac{\pi|z_2|^2}{b} \le 1 \right\}.
\] 
There are two simple Reeb orbits in $\partial E(a,b)$, given by the circles $(z_2=0)$ and $(z_1=0)$. These Reeb orbits have actions $a$ and $b$ respectively, and they are both Hopf orbits. If $a/b$ is rational then there are additional simple Reeb orbits, but these do not have minimal action. Using Step 2, we conclude that
\begin{equation}
\label{eqn:ellipsoid}
c_{\operatorname{Hopf}}(E(a,b)) = 
\mathcal{A}_{\operatorname{Hopf}}(E(a,b)) = \min(a,b).
\end{equation}
In particular, since $E(r,r)=B^4(r)$, we obtain $c_{\operatorname{Hopf}}(B^4(r))=r$.

To prove that $c_{\operatorname{Hopf}}(Z(R))=R$, let $\Omega\subset\R^4$ be a dynamically convex domain and suppose that there exists a symplectic embedding $\varphi:\Omega\hookrightarrow Z(R)$. We need to show that $\mathcal{A}_{\operatorname{Hopf}}(\Omega)\le R$. Given $\varepsilon>0$, there exists $C>>0$ such that $\varphi(\Omega)\subset E((1+\varepsilon)R,C)$. By Step~2, Monotonicity, and equation \eqref{eqn:ellipsoid}, it follows that
\[
\mathcal{A}_{\operatorname{Hopf}}(\Omega) = c_{\operatorname{Hopf}}(\Omega) \le c_{\operatorname{Hopf}}(E((1+\varepsilon)R,C)) \le (1+\varepsilon)R.
\]
Since $\varepsilon>0$ was arbitrary, we conclude that $\mathcal{A}_{\operatorname{Hopf}}(\Omega)\le R$ as desired.

\subsubsection*{Step~4} To complete the proof of Theorem~\ref{main_fast_capacity}, we need to show that if $X\subset\R^4$ is a dynamically convex star-shaped domain, then there exists $\gamma\in \mathcal{P}_{\rm Hopf}(\partial X)$ with $\CZ(\gamma)=3$ such that $\action(\gamma)=\action_{\rm Hopf}(X)$.

Let $\{\gamma_n\}_{n=1,\ldots}$ be a sequence in $\mathcal{P}_{\rm Hopf}(\partial X)$ satisfying $\mathcal{A}(\gamma_n) < e\mathcal{A}_{\operatorname{Hopf}}(X) \ \forall n$ and $\lim_{n\to \infty}\mathcal{A}(\gamma_n)=\mathcal{A}_{\operatorname{Hopf}}(X)$. We remark that by a compactness argument and the Arzel\`a-Ascoli theorem, one can pass to a subsequence so that $\gamma_n$ converges to a Reeb orbit with action $\mathcal{A}_{\operatorname{Hopf}}(X)$. However, this not enough to conclude  the proof of the theorem, because the limiting Reeb orbit might not be simple or have Conley-Zehnder index~$3$. We will instead need to take a limit of a sequence of Reeb orbits of different contact forms.

Let $\lambda$ denote the contact form on $\partial X$. 
We can construct a sequence of functions $f_n:\partial X \to (0,1)$ with $f_n\to 0$ in $C^\infty$ such that:
\begin{itemize}
\item The contact form $e^{f_n}\lambda$ is nondegenerate and dynamically convex up to action $e^3\mathcal{A}_{\operatorname{Hopf}}(X)$.
\item
The Reeb vector field of $e^{f_n}\lambda$ is tangent to $\gamma_n$, so that
the loop $\gamma_n$, suitably reparametrized, defines a periodic orbit of the Reeb flow of~$e^{f_n}\lambda$.
\end{itemize}
The existence of $f_n$ is elementary and somewhat standard, see~\cite[Section~6]{convex}. Specifically, Lemma~6.8 from~\cite{convex} contains this existence statement where the corresponding orbit is, in addition, assumed to have ${\rm CZ}=3$, but the perturbation argument is independent of this latter assumption. The idea is quite simple. If $\hat g_n:\partial X \to \R$ satisfies $\hat g_n|_{\gamma_n}=1$ and $d\hat g_n|_{\gamma_n}=0$, then $\gamma_n$ persists as a periodic Reeb orbit of $\hat g_n\lambda$. Moreover,~$\hat g_n$ can be chosen close to~$1$ in $C^\infty$, and so that its second derivatives along~$\gamma_n$ perturb the linearized dynamics and make~$\gamma_n$, together with all its iterates, nondegenerate periodic orbits. Then, by the usual result, one can $C^\infty$-slightly perturb $\hat g_n$ to $\tilde{g}_n$ so that $\tilde{g}_n\lambda$ is a nondegenerate contact form. Since $\gamma_n$ is a nondegenerate periodic Reeb orbit of $\hat g_n\lambda$, it gets slightly perturbed to a periodic Reeb orbit of $\tilde{g}_n\lambda$ which can, via a $C^\infty$-small contact isotopy, be brought back to $\gamma_n$. After this last transformation, the contact form assumes the form~$g_n\lambda$ for some~$g_n$ $C^\infty$-close to~$1$. Take $\Delta_n>0$ small so that $g_n+\Delta_n>1$ and define $f_n$ by $e^{f_n}=g_n+\Delta_n$.

The contact forms $e^{f_n}\lambda$ on $X$ correspond to star-shaped  domains $X_n$ such that $X \subset {\rm int}(X_n)$ and $\partial X_n \to \partial X$ in $C^\infty$ as hypersurfaces.
Note that
\begin{equation}
\label{eqn:soi}
\mathcal{A}(\gamma_n) < \mathcal{A}(\widetilde{\gamma}_n)< e^{\max(f_n)}\mathcal{A}(\gamma_n) < e\mathcal{A}(\gamma_n) < e^2\mathcal{A}_{\operatorname{Hopf}}(X).
\end{equation}
In particular,
\begin{equation}
\label{eqn:squeeze}
\lim_{n\to\infty} \mathcal{A}(\widetilde{\gamma}_n) = \mathcal{A}_{\operatorname{Hopf}}(X)
\end{equation}
since $\max(f_n) \to 0$.

Now for each $n$, choose a constant $\epsilon_n\in(0,\min(f_n))$. The contact form $e^{f_n-\epsilon_n}\lambda$ is nondegenerate and dynamically convex up to action $e^2\mathcal{A}_{\operatorname{Hopf}}(X)$, since it is a constant scaling of $e^{f_n}\lambda$. This contact form corresponds to a rescaling $X'_n$ of $X_n$ satisfying $X \subset {\rm int}(X'_n)$, $X'_n \subset {\rm int}(X_n)$, and $\partial X'_n \to \partial X$ in $C^\infty$. By Lemma~\ref{lemma_hard_step} and \eqref{eqn:soi}, there exists a suitable Hopf orbit $\gamma_n'\in\mathcal{P}_{\operatorname{Hopf}}(\partial X'_n)$ such that
\begin{equation}
\label{eqn:invokelemma}
\mathcal{A}(\gamma_n') \leq \mathcal{A}(\widetilde{\gamma}_n)
\end{equation}
and $\operatorname{CZ}(\gamma_n')=3$. An application of Lemma~\ref{lemma_easy_step} now gives the desired periodic orbit $\gamma \in \mathcal{P}_{\rm Hopf}(\partial X)$ satisfying ${\rm CZ}(\gamma)=3$ and $\mathcal{A}(\gamma) = \mathcal{A}_{\rm Hopf}(X)$.
%
\end{proof}

We now prepare for the proof of Lemma~\ref{lemma_hard_step}, which occupies the rest of Section~\ref{sec:preliminaries}.


\subsection{Reeb orbits, asymptotic operators, and Conley-Zehnder indices}
\label{sec:CZ}

Let $\lambda$ be a contact form on a 3-manifold $Y$.  Let $J$ be a complex structure on the vector bundle $\xi=\ker(\lambda)$ which is compatible with $d\lambda$. Given a Reeb orbit $\gamma$ with period $T$, the {\bf asymptotic operator} is the unbounded operator $L_{\gamma,J}$ on $L^2(\gamma(T\cdot)^*\xi)$ sending
\[
\eta \longmapsto -J\nabla_t\eta
\]
where $\nabla$ is the connection on $\gamma(T\cdot)^*\xi$ induced by the linearized Reeb flow.
One can check that this operator is self-adjoint when $L^2(\gamma(T\cdot)^*\xi)$ is equipped with the inner-product induced by $(d\lambda,J)$.

The following are some key facts about the spectral theory of the asymptotic operator, which follow from perturbation theory as explained in~\cite[\S3]{props2}.
The asymptotic operator has discrete spectrum, consisting of eigenvalues whose geometric and algebraic multiplicities coincide, and the spectrum accumulates at $\pm\infty$. 
The eigenvectors are nowhere vanishing sections of $\gamma(T\cdot)^*\xi$ since they solve linear ODEs. Hence they have well-defined winding numbers with respect to a $d\lambda$-symplectic trivialization $\tau$ of $\gamma(T\cdot)^*\xi$. The winding number of an eigenvector depends only on the eigenvalue. So to an eigenvalue $\nu$ one can associate the winding number
\[
\wind_\tau(\nu) \in \Z.
\]
For every $k\in\Z$ there are precisely two eigenvalues counted with multiplicity satisfying $\wind_\tau = k$, and $\nu_1\leq\nu_2 \Rightarrow \wind_\tau(\nu_1)\leq\wind_\tau(\nu_2)$. For $\delta \in \R$ we set
\begin{equation*}
\begin{aligned}
\alpha^{<\delta}_\tau(\gamma) &= \max \ \{ \wind_\tau(\nu) \mid \nu \ \text{eigenvalue}, \ \nu<\delta \} \\
\alpha^{\geq\delta}_\tau(\gamma) &= \min \ \{ \wind_\tau(\nu) \mid \nu \ \text{eigenvalue}, \ \nu\geq\delta \} \\
p^{\delta}(\gamma) &= \alpha^{\geq\delta}_\tau(\gamma) - \alpha^{<\delta}_\tau(\gamma) \in \{0,1\} \, .
\end{aligned}
\end{equation*}
These numbers do not depend on $J$.

We define the {\bf constrained Conley--Zehnder index\/}
\begin{equation}
\label{eqn:ccz}
\CZ^\delta_\tau(\gamma) = 2 \alpha^{<\delta}_\tau(\gamma) + p^{\delta}(\gamma).
\end{equation}
We note that $\CZ_\tau^\delta$ is a lower semi-continuous function of the asymptotic operator $L_{\gamma,J}$ with respect to the $C^\infty$ norm. We define the {\bf Conley-Zehnder index\/} $\CZ_\tau(\gamma)=\CZ_\tau^0(\gamma)$, i.e.\ by taking $\delta=0$ in \eqref{eqn:ccz}. 

If $\Omega\subset\R^4$ is a star-shaped domain and $Y=\partial \Omega$, then unless otherwise stated we choose $\tau$ to be a global trivialization of $\xi=\ker(\lambda_0|_Y)$, and for any Reeb orbit $\gamma\in\P(Y)$, we define
\[
\CZ(\gamma)=\CZ_\tau^0(\gamma).
\]


\subsection{Pseudoholomorphic curves in symplectizations}
\label{ssec_pseudo_hol_curves}

Let $J$ be a compatible complex structure on the symplectic vector bundle~$(\xi,d\lambda)$. As in \cite{93}, consider an almost complex structure $\jtil$ defined on $\R\times Y$ by
\begin{equation}\label{R_inv_J}
\jtil\partial_a = R_\lambda, \qquad \jtil|_{\xi} = J
\end{equation}
where $R_\lambda$ and $\xi$ are seen as $\R$-invariant objects in $\R\times Y$. Then $\jtil$ is $\R$-invariant. We say that an almost complex structure $\jtil$ on $\R\times Y$ constructed this way is {\bf $\lambda$-compatible\/}.

Consider a closed Riemann surface $(S,j)$, a finite set $\Gamma \subset S$ of ``punctures'', and a holomorphic map
\[
\util = (a,u) : (S\setminus\Gamma,j) \longrightarrow (\R\times Y,\jtil)
\]
Define the {\bf Hofer energy\/}
\[
E(\util) = \sup_{\phi} \int_{S\setminus\Gamma} \util^*d(\phi\lambda).
\]
where the supremum is taken over the set of $\phi:\R \to [0,1]$ satisfying $\phi'\geq0$. If the Hofer energy $E(\util)<\infty$, then we say that $\util$ is a {\bf finite energy curve\/}.

A puncture $z\in\Gamma$ is {\bf positive\/} or {\bf negative\/} if $a(w)\to+\infty$ or $a(w)\to-\infty$ when $w\to z$, respectively. The puncture $z$ is {\bf removable} if $\limsup |a(w)| < \infty$ when $w\to z$. It turns out that for a finite energy curve, every puncture is positive, negative or removable, and $\util$ can be smoothly extended across a removable puncture; see \cite{93}. So we can assume without loss of generality that $\Gamma$ does not contain removable punctures, and we denote by $\Gamma_+$ and $\Gamma_-$ the sets of positive and negative punctures, respectively.

Let $z\in\Gamma$ and let $K\subset S$ be a conformal disk centered at $z$, meaning that there is a biholomorphism $\varphi : (K,j,z) \to (\D,i,0)$. Then $K\setminus\{z\}$ admits positive holomorphic polar coordinates $(s,t) \in [0,+\infty) \times \R/\Z$ defined by $(s,t) \simeq \varphi^{-1}(e^{-2\pi(s+it)})$, and negative holomorphic polar coordinates $(s,t) \in (-\infty,0] \times \R/\Z$ defined by $(s,t) \simeq \varphi^{-1}(e^{2\pi(s+it)})$.

As at the beginning of \S\ref{sec:preliminaries}, we say that $\lambda$ is nondegenerate up to action $C$ if all Reeb orbits of period less than $C$ are nondegenerate. If all Reeb orbits are nondegenerate, we say that $\lambda$ is nondegenerate.

\begin{theorem}[\cite{props1}]
Suppose that $\lambda$ is nondegenerate up to action $C$ and that $z$ is a nonremovable puncture of a finite energy curve $\util=(a,u)$ in $(\R\times Y,\jtil)$ with Hofer energy $E(\util)\leq C$. Let $(s,t)$ be positive holomorphic polar coordinates near $z$. There exist a Reeb orbit $\gamma$ with period $T$ and $d\in\R$ such that $u(s,t) \to \gamma(\epsilon Tt)$ in $C^\infty(\R/\Z,Y)$ as $s\to+\infty$, where $\epsilon = \pm1$ is the sign of the puncture. 
\end{theorem}

The Reeb orbit $\gamma$ above is called the {\bf asymptotic limit} of $\util$ at $z$.

Let $\gamma$ be a Reeb orbit of $\lambda$ with period $T$. We say that $\gamma$ is {\bf multiply covered\/} if $\gamma$ is not simple, i.e.\ if $\gamma(T')=\gamma(0)$ for some $T'\in (0,T)$. In that case, the {\bf covering multiplicity\/} of $\gamma$ is $k=T/T_0$ where $T_0$ is the minimal period.

Consider the space $\R/\Z\times\C$ equipped with coordinates $(\vartheta,z=x_1+ix_2)$ and contact form $\beta_0 = d\vartheta + x_1dx_2$. 

\begin{definition}
Let $\gamma$ be a Reeb orbit of $\lambda$ with period $T$ and converging multiplicity $k$. A {\bf Martinet tube\/} for $\gamma$ is a smooth diffeomorphism $\Psi:\mathcal{N} \to\R/\Z\times \D$ defined on a smooth compact neighborhood $\mathcal{N}$ of $\gamma(\R)$ such that: 
\begin{itemize}
\item $\Psi(\gamma(T\vartheta/k)) = (\vartheta,0)$ for all $\vartheta \in \R/\Z$.
\item On $\mathcal{N}$ we have $\lambda = \Psi^*(g\beta_0)$, where $g : \R/\Z\times \D \to (0,+\infty)$ is smooth and satisfies $g(\vartheta,0) = T/k$, $dg(\vartheta,0) = 0$ for all $\vartheta \in \R/\Z$.
\end{itemize}
\end{definition}

For the existence of Martinet tubes the reader is referred to~\cite[Lemma~2.3]{props1}.

\begin{theorem}[\cite{props1,sie_CPAM}]
\label{thm_asymptotics}
Suppose that $\lambda$ is nondegenerate up to action $C>0$ and that $z$ is a nonremovable puncture of sign $\epsilon = \pm1$ of a finite energy curve $\util=(a,u)$ with Hofer energy $E(\util)\leq C$. Let $(s,t)$ be positive or negative holomorphic polar coordinates at $z$ when $\epsilon=+1$ or $\epsilon=-1$, respectively. Let $\Psi:\mathcal{N} \to\R/\Z\times \D$ be a Martinet tube for the asymptotic limit $\gamma$ of $\util$ at $z$, and let $s_0\gg1$ such that $|s|\geq s_0 \Rightarrow u(s,t) \in \mathcal{N}$. Write $\Psi(u(s,t)) = (\vartheta(s,t),z(s,t))$ for $|s|\geq s_0$. By applying a rotation we can assume $u(s,0) \to \gamma(0)$ as $\epsilon s\to+\infty$. 

If $z(s,t)$ does not vanish identically then the following holds. There exists $r>0$ and an eigenvalue $\nu$ of the asymptotic operator $L_{\gamma,J}$ satisfying $\epsilon\nu<0$, such that:
\begin{itemize}
\item There exist $c,d\in\R$ and a lift $\tilde\vartheta:\R\times\R \to \R$ of $\vartheta(s,t)$ such that $$ \lim_{\epsilon s\to+\infty} \sup_{t\in\R/\Z} e^{r\epsilon s} \left( |D^\beta[a(s,t)-Ts-c]| + |D^\beta[\tilde\vartheta(s,t)-kt]| \right) = 0  $$ holds for every partial derivative $D^\beta = \partial^{\beta_1}_s\partial^{\beta_2}_t$, where $k$ is the covering multiplicity of $\gamma$.
\item There exists an eigenvector of $\nu$, represented as a nowhere vanishing vector field $v(t)$ in the frame $\{\partial_{x_1},\partial_{x_2}\}$ along $\gamma$, such that
\begin{equation*}
z(s,t) = e^{\nu s} (v(t)+R(s,t))
\end{equation*}
for some $R(s,t)$ satisfying $|D^\beta R(s,t)| \to 0$ in $C^0(\R/\Z)$ as $\epsilon s\to+\infty$, for every partial derivative $D^\beta = \partial^{\beta_1}_s\partial^{\beta_2}_t$.
\end{itemize}
\end{theorem}

The eigenvalue $\nu$ provided by Theorem~\ref{thm_asymptotics} is called the {\bf asymptotic eigenvalue\/} of $\util$ at the puncture $z$.

\begin{remark}
\label{rem:trivial_puncture}
The alternative $z(s,t) \equiv 0$ can be expressed independently of coordinates as saying that the end of the domain of $\util$ corresponding to the puncture is mapped into the ``trivial cylinder'' $\R\times\gamma$ where $\gamma$ is the asymptotic limit. In this case we say that $\util$ has {\bf trivial\/} asymptotic behavior at the puncture. Otherwise, the asymptotic behavior is said to be {\bf nontrivial\/} at the puncture.
\end{remark}

We now recall some topological invariants introduced in \cite{props2}. Let $\util=(a,u)$ be a finite energy curve in $(\R\times Y,\jtil)$. Assume that $\lambda$ is nondegenerate up to action $E(\util)$. Let 
\begin{equation}
\label{proj_ctct_str}
\pi_\lambda : TY \longrightarrow \xi
\end{equation}
denote the projection along the Reeb vector field $R_\lambda$. It can be shown using the Carleman similarity principle that if $\pi_\lambda \circ du$ does not vanish identically, then its zeros are isolated and count positively. Theorem~\ref{thm_asymptotics} further implies that there are finitely many zeros in this case. So if $\util$ is nontrivial, then the algebraic count of zeros of $\pi_\lambda\circ du$ is a well-defined non-negative number, which following \cite{props2} we denote by $\wind_\pi(\util)$. The inequality ${\rm wind}_\pi(\tilde{u}) \geq 0$ could be thought of as some kind of ``infinitesimal'' positivity of intersections. This motivates and clarifies a bit the meaning and the usefulness of the invariant ${\rm wind}_\pi$.

Fix a $d\lambda$-symplectic trivialization $\tau$ of $u^*\xi$. Let $z$ be a puncture of $\util$ with asymptotic limit $\gamma$ with period $T$. Let $\wind_\infty(\util,z,\tau) = \wind_\tau(\nu) \in \Z$, where $\nu$ is the asymptotic eigenvalue of $\util$ at $z$. Finally, define
\[
\wind_\infty(\util) = \sum_{z\in\Gamma_+} \wind_\infty(\util,z,\tau) - \sum_{z\in\Gamma_-}  \wind_\infty(\util,z,\tau).
\]
Note that $\wind_\infty(\util)$ does not depend on the choice of trivialization $\tau$ of $u^*\xi$. In fact, since $\pi_\lambda\circ du$ is a section of the complex line bundle $T^{1,0}S\otimes \xi$, it follows from degree theory that
\begin{equation}
\label{identity_winds}
\wind_\pi(\util) = \wind_\infty(\util) - \chi(S\setminus\Gamma).
\end{equation}
See~\cite[Proposition~5.6]{props2}.

Denote by $(\overline{\C} = \C \cup \{\infty\},i)$ the Riemann sphere. For the next two definitions consider a finite-energy plane $\util = (a,u) : (\C,i) \to (\R\times Y,\jtil)$, and assume that $\lambda$ is nondegenerate up to action $E(\util)$. By Stokes theorem, $\infty$ must be a positive puncture, and the similarity principle implies that $\int_\C u^*d\lambda>0$. 

\begin{definition}
[\cite{fast}]
\label{def_fast_plane}
The plane $\util$ is {\bf fast} if $\wind_\infty(\util)=1$.
\end{definition}

\begin{definition}\label{def_cov_plane}
The {\bf asymptotic covering multiplicity} ${\rm cov}(\util)$ of the plane $\util$ is the covering multiplicity of its asymptotic limit.
\end{definition}

\begin{lemma}\label{lemma_fast_symplectisations}
If $\util=(a,u)$ is a fast plane, then $\util$ is somewhere injective and the map $u:\C\to Y$ is an immersion which is transverse to $R_\lambda$.
\end{lemma}

\begin{proof}
Since $\wind_\pi(\util)\ge 0$, it follows from~\eqref{identity_winds} that $\wind_\pi(\util)=0$, and hence $u$ is an immersion which is transverse to $R_\lambda$. If $\util$ is not somewhere injective then it covers another plane via a polynomial map of degree~$\geq2$, and the covering map must have at least one critical point, contradicting the fact that $u$ is an immersion.
\end{proof}


\subsection{Pseudoholomorphic curves in strong symplectic cobordisms}
\label{ssec_curves_cobodisms}

Consider a compact symplectic $4$-manifold  $(W,\omega)$ for which there exists a Liouville vector field $V$ near $Y = \partial W$ transverse to $Y$. The $1$-form $\alpha = i_V\omega$ defined near $Y$ satisfies $d\alpha=\omega$ and restricts to a contact form on $Y$ which we denote by $\lambda$. We denote the associated contact structure by $\xi = \ker \lambda \subset TY$. The boundary splits as $Y = Y^+ \sqcup Y^-$, where $V$ points out along $Y^+$ and points in along $Y^-$. The orientation of $Y$ as the boundary of $W$ agrees with the contact orientation on $Y^+$ and disagrees with it on $Y^-$. We call $(W,\omega)$ a {\bf strong symplectic cobordism} that is {\bf convex} at $(Y^+,\lambda)$ and {\bf concave} at $(Y^-,\lambda)$. The local flow $\varphi^t_V$ of $V$ defines, for $\varepsilon>0$ small enough, diffeomorphisms
\begin{equation}
\Phi^+ : (-\varepsilon,0] \times Y^+ \longrightarrow \mathcal{U}^+, \quad \Phi^- : [0,\varepsilon) \times Y^- \longrightarrow \mathcal{U}^-, \quad \Phi^\pm(a,p) = \varphi^a_V(p)
\end{equation}
onto neighborhoods $\mathcal{U}^\pm$ of $Y^\pm$ in $W$. We have $(\Phi^\pm)^*\alpha = e^a\lambda$, where here we still write $\lambda$ for the pull-back of $\lambda$ under the projection $\R \times Y \to Y$. The symplectic form pulls back to
\[
\left(\Phi^\pm\right)^*\omega = d(e^a\lambda) = e^a ( da \wedge \lambda + d\lambda ).
\]

The {\bf symplectic completion\/} of $(W,\omega)$ is the symplectic manifold $(\overline{W},\overline{\omega})$ defined by
\begin{equation}
\overline{W} = \left(\left(\left(-\varepsilon,+\infty\right) \times Y^+\right) \sqcup W \sqcup \left(\left(-\infty,\varepsilon\right) \times Y^- \right)\right) / \sim
\end{equation}
where points are identified according to
\begin{equation*}
\mathcal{U}^+ \ni \Phi^+(a,p)   \sim (a,p) \in (-\varepsilon,0] \times Y^+, \quad\quad  \mathcal{U}^- \ni \Phi^-(a,p)  \sim (a,p) \in [0,\varepsilon) \times Y^-
\end{equation*}
and the symplectic form $\overline{\omega}$ is defined by
\[
\overline{\omega} = \omega  \ \text{on} \ W \qquad\qquad \overline{\omega} = d(e^a\lambda) \ \text{on} \ (-\varepsilon,+\infty) \times Y^+ \ \text{and on} \ (-\infty,\varepsilon) \times Y^-.
\]

Fix a $\lambda$-compatible almost complex structure $\jtil$ on $\R\times Y$. 
Consider an almost complex structure $\jbar$ on $\overline{W}$ that is $\omega$-compatible on $W$, and that agrees with $\jtil$ on $[0,+\infty)\times Y^+ \cup (-\infty,0] \times Y^-$ as in~\eqref{R_inv_J}. It follows that $\jbar$ is $\overline{\omega}$-compatible. As in \cites{93,sft_comp} we consider a closed Riemann surface $(S,j)$, a finite set $\Gamma \subset S$ and a holomorphic map
\[
\util:(S\setminus\Gamma,j) \longrightarrow (\overline{W},\jbar).
\]
We assume the finite energy condition $0< E(\util)<\infty$, where the energy $E(\util)$ is defined by
\[
\begin{split}
E(\util) = &\int_{\util^{-1}(W)} \util^*\omega
\\
&+ \ \sup_{\phi} \left\{ \int_{\util^{-1}([0,+\infty)\times Y^+)} \util^*d(\phi\lambda) + \int_{\util^{-1}((-\infty,0] \times Y^-)} \util^*d(\phi\lambda) \right\}.
\end{split}
\]
Here the supremum is taken over the set of smooth $\phi : \R \to [0,1]$ satisfying $\phi'\geq0$. The finite energy condition implies that punctures are either removable or behave like the nonremovable punctures as explained in \S\ref{ssec_pseudo_hol_curves}. More precisely, if $(s,t)$ are positive holomorphic polar coordinates at a positive puncture $z\in\Gamma$ then $\util(s,t) \in [0,+\infty)\times Y^+$ when $s\gg1$, and if one writes $\util=(a,u)$ for $s\gg1$ then $a(s,t) \to+\infty$ as $s\to+\infty$. There is a similar behavior at negative punctures. Moreover, if $\lambda$ is nondegenerate up to action $E(\util)$, then there is an asymptotic limit $P$ at a nonremovable puncture, and all the conclusions of Theorem~\ref{thm_asymptotics} hold. In particular, we can talk about nontrivial versus trivial asymptotic behavior and asymptotic eigenvalues, see Remark~\ref{rem:trivial_puncture}. 

Now let $\util = (a,u) :(\C,i) \to (\W,\jbar)$ be a finite energy plane, where $\infty$ is  a positive puncture\footnote{Note that for a general $(W,\omega)$ the puncture of a plane might be negative. However it will always be a positive puncture when $\omega$ is exact.}.  Take $R>0$ large enough so that $u(\C\setminus B_R(0)) \subset [0,+\infty) \times Y^+$. Then there is a splitting 
\begin{equation}
\util^*T\W|_{\C\setminus B_R(0)} = (\util|_{\C\setminus B_R(0)})^*\left<\partial_a,R_\lambda\right> \oplus (u|_{\C\setminus B_R(0)})^*\xi
\end{equation}
where $\left<\partial_a,R_\lambda\right>$ denotes the subbundle of $T([0,+\infty)\times Y^+)$ spanned by $\{\partial_a,R_\lambda\}$. There is a $d\lambda$-symplectic frame $\{e_1,e_2\}$ of $(u|_{\C\setminus B_R(0)})^*\xi$ such that $\{\partial_a,R_\lambda,e_1,e_2\}$ extends to a $\overline{\omega}$-symplectic frame of $\util^*T\W$. Just take, for instance, a frame that extends to a global frame of $u^*\xi$. This frame is unique up to homotopy: any two such frames would have relative Maslov index equal to zero along any loop in their domain of definition~$\C$. If $\lambda$ nondegenerate up to action $E(\util)$, then $\util$ has an asymptotic limit $\gamma$ in $Y^+$, of period~$T>0$, and it follows from the asymptotic behavior that, up to homotopy, $\{e_1,e_2\}$ induces a $d\lambda$-symplectic trivialization $\tau_{\util}$ of $\gamma(T\cdot)^*\xi$. Define
\begin{equation}
\CZ(\util) = \CZ_{\tau_{\util}}(\gamma),
\end{equation}
If the asymptotic behavior of $\util$ is nontrivial, define
\begin{equation}
\wind_\infty(\util) = \wind_{\tau_{\util}}(\nu)
\end{equation}
where $\nu$ is the asymptotic eigenvalue of $\util$ at the puncture $\infty$. If the asymptotic behavior is trivial then we define
\begin{equation}
\wind_\infty(\util) = -\infty.
\end{equation}
Under the above nondegeneracy assumption, we define a plane with a positive puncture to be {\bf fast} if $\wind_\infty(\util)\leq1$. Unlike the case of symplectizations, here a fast plane need not be immersed or somewhere injective; the proof of Lemma~\ref{lemma_fast_symplectisations} does not carry over because $\operatorname{wind}_\pi$ does not make sense in this context. As before, we define ${\rm cov}(\util)$ to be the covering multiplicity of the asymptotic limit. If ${\rm cov}(\util)=1$ then $\util$ is somewhere injective. 


\subsection{The main compactness argument}

We now prove Lemma~\ref{lemma_hard_step}. Before diving into the argument, let us sketch the main steps. 

Firstly, we use the lemma holds under the assumptions that $\varphi(\Omega_1) \subset {\rm int}(\Omega_2)$ and $\Omega_2$ is nondegenerate up to action~$C$, to show the version of the lemma when these assumptions are dropped. This is Step~0. The remaining argument is to show the lemma when $\varphi(\Omega_1) \subset {\rm int}(\Omega_2)$ and $\Omega_2$ is nondegenerate up to action~$C$. The idea is: (1)~consider the symplectic cobordism from $\partial\Omega_2$ to $\varphi(\partial\Omega_1)$ induced by $\Omega_2$, which has a convex boundary at $\partial\Omega_2$ and a concave boundary at $\varphi(\partial\Omega_1)$, (2)~use existence results for fast planes from~\cite{fast,hryn_jsg} to see that some moduli space of fast planes in this cobordism is nonempty, (3)~``push'' the fast planes in this moduli space ``inside'' the concave boundary $\varphi(\partial\Omega_1)$, (4)~show that this family of planes SFT-converges to cylindrical building consisting of a cylinder in the cobordism, and a fast plane in the symplectization of $\varphi(\partial\Omega_1)$ which is asymptotic to the desired Hopf orbit in~$\varphi(\partial\Omega_1)$.

\subsubsection*{Step~0}
First assume that Lemma~\ref{lemma_hard_step} holds under the additional assumptions that $\varphi(\Omega_1) \subset {\rm int}(\Omega_2)$ and that $\Omega_2$ is nondegenerate up to action~$C$. In this case the lemma provides the orbit $\gamma \in \P_{\rm Hopf}(\partial\Omega_1)$ with a bound $\action(\gamma) < \action_{\rm Hopf}(\Omega_2)$. Using this, we will now prove the version of Lemma~\ref{lemma_hard_step} where these extra assumptions are dropped.

Consider a sequence $\gamma_n \in \P_{\rm Hopf}(\partial\Omega_2)$ satisfying $\action(\gamma_n) \to \action_{\rm Hopf}(\Omega_2)$. Choose $\epsilon_n>0$, $\epsilon_n \to 0$ an arbitrary sequence. We follow a scheme similar to Step~4 in the proof of Theorem~\ref{main_fast_capacity}. Note that $\tilde\gamma_n = e^{\epsilon_n}\gamma_n$ is a Hopf orbit in~$e^{\epsilon_n}\partial \Omega_2$. Let $\Omega^n_2$ be a small $C^\infty$-perturbation of $e^{\epsilon_n}\Omega_2$ such that:
\begin{itemize}
\item $\Omega^n_2$ is nondegenerate and dynamically convex up to action~$C$.
\item $\tilde\gamma_n \subset \partial\Omega^n_2$.
\item $\partial \Omega^n_2 \to \partial \Omega_2$ in $C^\infty$ as hypersurfaces.
\item $\Omega_2 \subset {\rm int}(\Omega^n_2)$.
\end{itemize}
In particular, $\tilde\gamma_n \in \P_{\rm Hopf}(\partial\Omega^n_2)$. Choose $\delta_n \in (0,\epsilon_n)$ such that $\Omega_2 \subset {\rm int}(e^{-\delta_n}\Omega^n_2)$. Note that $e^{-\delta_n}\tilde\gamma_n = e^{\epsilon_n-\delta_n}\gamma_n$ belongs to $\P_{\rm Hopf}(e^{-\delta_n}\partial\Omega_2^n)$.
By the assumed version of Lemma~\ref{lemma_hard_step}, there exist periodic orbits $\hat\gamma_n \in \P_{\rm Hopf}(\partial\Omega_1)$ satisfying
\[
\action(\hat\gamma_n) < \action_{\rm Hopf}(e^{-\delta_n}\Omega^n_2) \leq e^{\epsilon_n-\delta_n}\action(\gamma_n) \to \action_{\rm Hopf}(\Omega_2), \qquad {\rm CZ}(\hat\gamma_n) = 3.
\]
A direct application of Lemma~\ref{lemma_easy_step} to the sequence $\hat\gamma_n$ in $\Omega_1$ establishes the existence of $\gamma \in \P_{\rm Hopf}(\Omega_1)$ satisfying $\action(\gamma) \leq \action_{\rm Hopf}(\Omega_2)$ and $\CZ(\gamma)=3$. The proof of Step~0 is complete.

\bigskip

With Step~0 established, we now proceed with the proof of Lemma~\ref{lemma_hard_step} under the additional assumptions that $\Omega_2$ is nondegenerate up to action~$C > \action_{\rm Hopf}(\Omega_2)$, and that $\varphi(\Omega_1) \subset {\rm int}(\Omega_2)$. In this case we wish to find $\gamma \in \P_{\rm Hopf}(\Omega_1)$ satisfying $\CZ(\gamma)=3$ and $\action(\gamma) < \action_{\rm Hopf}(\Omega_2)$.

\subsubsection*{Step 1}
We begin with some geometric setup.

Consider the compact manifold
\[
W = \Omega_2 \setminus \varphi(\Omega_1\setminus\partial\Omega_1),
\]
with symplectic form $\omega=\omega_0$. We have
\[
\partial W = Y^+ - Y^-
\]
where $Y^+ = \partial\Omega_2$ and $Y^- = \varphi(\partial\Omega_1)$. The radial vector field on $\Omega_2$ is Liouville, transverse and outward pointing at $Y^+$. The pushforward of the radial vector field on $\Omega_1$ by $\varphi$ extends to a Liouville vector field defined on a neighbourhood of $Y^-$ in $W$, where it is transverse to $Y^-$ and inward pointing. Hence, we can define the symplectic completion $(\W,\overline{\omega})$ of $(W,\omega)$ as in~\S~\ref{ssec_curves_cobodisms}. As usual, let $\lambda$ denote the contact form on $Y^+ \cup Y^-$. By assumption, $\lambda$ is, up to action $C$, dynamically convex and nondegenerate. Write $\xi = \ker\lambda$. Fix a global $d\lambda$-symplectic trivialization~$\tau$ of~$\xi$. 

Choose any complex structure $J$ on $\xi$ which is $d\lambda$-compatible, and let $\jtil$ be the corresponding $\lambda$-compatible almost complex structure on $\R\times Y$ as in equation \eqref{R_inv_J}. Let $\jbar$ be an almost complex structure on $\W$ which is $\omega$-compatible on $W$ and agrees with $\jtil$ on $(-\infty,0]\times Y^- \cup [0,+\infty) \times Y^+$. In particular, $\jbar$ is $\overline{\omega}$-compatible. We need to choose the extension $\jbar$ generically, in a way to be specified in Step 4 below.

\subsubsection*{Step 2}
Let $\gamma_+$ be a simple $\lambda$-Reeb orbit in~$Y^+ = \partial\Omega_2$ with period $T\leq C$. Let $\Mfast(\gamma_+,\jbar)$ denote the moduli space of equivalence classes of fast and embedded finite energy planes $\util:(\C,i) \to (\W,\jbar)$ asymptotic to $\gamma_+$. Two planes $\util_0,\util_1$ are declared equivalent if there exist $(A,B) \in\C^*\times\C$ such that $\util_1(Az+B) = \util_0(z)$ for all $z\in\C$. We now show that $\Mfast(\gamma_+,\jbar)$ is naturally endowed with the structure of a two-dimensional smooth manifold.

Fix $\delta<0$ in the interval between eigenvalues of the asymptotic operator $L_{\gamma_+,\jtil}$ satisfying $\wind_\tau=1$ and $\wind_\tau=2$. This is possible since it follows from dynamical convexity that $\CZ^0_\tau(\gamma_+) \geq 3$. Note that $\alpha^{<\delta}_\tau(\gamma_+)=1$ and $\CZ^\delta_\tau(\gamma_+) = 3$. Following \cite{props3}, one can build a Fredholm theory with weight~$\delta$ based on sections of the normal bundle. In more detail, let $\util = (a,u)$ represent an arbitrary element of $\Mfast(\gamma_+,\jbar)$. Using the asymptotic behavior from Theorem~\ref{thm_asymptotics} and the fact that $\util$ is assumed to be an embedding (an immersion would suffice), one can consider a $\jbar$-invariant normal bundle $N$ of $\util$ that agrees with $u^*\xi$ near the puncture $\infty$. Then one looks at sections of the normal bundle that decay faster then $e^{\delta s}$ near $\infty$, where a neighbourhood of $\infty$ is parametrised by $(s,t) \simeq e^{2\pi(s+it)}$, $s\gg1$. One can use here Sobolev or H\"older spaces. The linearization of the Cauchy-Riemann equations at $\util$ determines a Fredholm operator $D_{\util}$ on this space of sections with index
\begin{equation*}
{\rm ind}_\delta(\util) = \CZ^\delta_\tau(\gamma_+) - 1 = 3-1 = 2.
\end{equation*}

\begin{lemma}
\label{lemma_autom_transv}
The operator $D_{\util}$ is surjective.
\end{lemma}


\begin{proof}
The contact structure along the plane can be identified, near the puncture, with the normal bundle of the plane. This follows from the asymptotic formula. We proceed assuming this identification is made. Consider $x+iy$ global coordinates on $\C$, and $(s,t) \in \mathbb{R} \times \mathbb{R}/\mathbb{Z}$ coordinates on $\C \setminus \{0\}$ defined via the biholomorphism $(s,t) \simeq e^{2\pi(s+it)}$. A symplectic trivialization of the normal bundle~$N$ can be homotoped to a trivialization that extends over a $d\lambda$-symplectic trivialization $\tau_N$ of $x(T\cdot)^*\xi$.  A nowhere vanishing section $\zeta_N$ of $x(T\cdot)^*\xi$ such that $\wind_{\tau_N}(\zeta_N)=0$ extends to a nowhere vanishing section of~$N$ still denoted by $\zeta_N$. Hence, $\{d\util \cdot \partial_x,\zeta_N\}$ is a complex frame of $\util^*TW$. A nowhere vanishing section $\zeta_1$ of $x(T\cdot)^*\xi$ such that $\wind_{\tau}(\zeta_1)=0$ extends to a nowhere vanishing section of~$\util^*\xi$ still denoted by $\zeta_1$. Hence, $\{\partial_a,\zeta_1\}$ is a complex frame of $\util^*TW$. The asymptotic behavior tells us that $\partial_a$ is roughly $d\util \cdot \partial_s$, up to a positive constant, near the puncture. Since~$\C$ is contractible, $$ 0 = \wind(d\util \cdot \partial_s,d\util \cdot \partial_x) + \wind(\zeta_1,\zeta_N) = 1 + \wind(\zeta_1,\zeta_N). $$ This shows that $\tau_N$ is ``one turn ahead'' of $\tau$. Hence, a nowhere vanishing section of $x(T\cdot)^*\xi$ with winding number equal to $0$ with respect to $\tau_N$ will have winding number equal to $+1$ with respect to~$\tau$. In particular,
\[
\alpha^{<\delta}_{\tau_N}(\gamma_+) = \alpha^{<\delta}_{\tau}(\gamma_+) - 1 = 1-1=0.
\]
A section $\zeta$ in the kernel of $D_{\util}$ satisfies a Cauchy-Riemann type equation. Hence, either $\zeta \equiv 0$, or its zeros are isolated and count positively to the total algebraic count of zeros. Moreover, if $\zeta \not \equiv 0$ then its asymptotic behavior at $\infty$ is governed by an eigenvector $e(t)$ of the asymptotic operator $L_{\gamma_+,\jtil}$ associated to an eigenvalue $\nu<0$ in a similar manner to what is described by Theorem~\ref{thm_asymptotics}: $\zeta(s,t) \sim e^{\nu s}(e(t)+\varepsilon(s,t))$, $\sup_t|\varepsilon(s,t)| \to 0$ as $s\to+\infty$. One can justify this asymptotic behavior using~\cite[Theorem~A.1]{sie_CPAM}. Hence $\zeta$ has only finitely many zeros, and by the argument principle the total count of zeros of~$\zeta$ is nonnegative and equal to $\wind_{\tau_N}(\nu)$. The exponential decay $e^{\delta s}|\zeta(s,t)| \to 0$ forces $\nu<\delta$. Thus $\wind_{\tau_N}(\nu) \leq \alpha^{<\delta}_{\tau_N}(\gamma_+) = 0$, and we conclude that~$\zeta$ never vanishes if it does not vanish identically. If $D_{\util}$ is not surjective then its kernel has dimension $\geq 3$ since the index is $2$, so we would find three linearly independent sections in the kernel, but this is impossible: a linear combination of them would vanish at some point because the normal bundle has rank two. This proves the desired surjectivity.
\end{proof}

The automatic transversality proved above endows $\Mfast(\gamma_+,\jbar)$ with the structure of a two-dimensional manifold. Note that since the Fredholm theory is modelled on a space of sections with exponential decay faster than $\delta$, the asymptotic eigenvalue of the obtained nearby planes is always $<\delta$, implying that $\wind_\infty \leq \alpha^{<\delta}_\tau = 1$ for the nearby planes.

In preparation for compactness arguments below, we note that one can use results on cylinders of small area, such as those proven in~\cite{small_area}, together with the assumed nondegeneracy of $\gamma_+$, to conclude that the topology $\Mfast(\gamma_+,\jbar)$ inherits from the functional analytic set-up agrees with the topology induced by $C^\infty_{\rm loc}$-convergence. This is the property of ``completeness'' in \cite{props3}.

\subsubsection*{Step 3}
Suppose now that $\util$ is an embedded fast plane in $\overline{W}$ asymptotic to $\gamma_+$. Let $\mathcal{Y}\subset \Mfast(\gamma,\jbar)$ denote the connected component of $[\util]$.
 For each $L\geq0$ consider the compact subset of $\W$ defined by
\[
E_L = \left([-L,0]\times Y^-\right)  \sqcup  W  \sqcup  \left([0,L]\times Y^+\right).
\]
Define
\[
\mathcal{Y}_L = \{ [\vtil] \in \mathcal{Y} \mid \vtil(\C) \cap E_L \neq \emptyset, \ \vtil(\C) \cap \left((-\infty,-L) \times Y^-\right) = \emptyset \}.
\]
The goal of this step is to show that $\mathcal{Y}_L$ is compact. 

Let $[\vtil_n] \in \mathcal{Y}_L$ be an arbitrary sequence. By the SFT compactness theorem \cite{sft_comp} we find a subsequence, still denoted by $\vtil_n$, that SFT converges to a holomorphic building ${\bf u}$. The bottom level of the building is a (possibly nodal) curve in $(\W,\jbar)$, but in principle there might be more levels in $(\R\times Y^+,\jtil)$. Any level in $\R\times Y^+$ must have at least one irreducible component which is not a trivial cylinder $\R\times\gamma$ where $\gamma$ is a Reeb orbit.

We first prove that~${\bf u}$ has only the bottom level. Suppose by contradiction that~${\bf u}$ has levels in $\R\times Y^+$. Then the top level is a finite energy punctured sphere
\begin{equation*}
\util_+ = (a_+,u_+) : (\C\setminus\Gamma,i) \longrightarrow (\R\times Y^+,\jtil)
\end{equation*}
with precisely one positive puncture (at $\infty$) where it is asymptotic to $\gamma_+$, and negative punctures at $\Gamma$. If $\int u_+^*d\lambda=0$, then $\util_+$ is the trivial cylinder over~$\gamma_+$ since~$\gamma_+$ is simply covered, contradicting the nontriviality in the previous paragraph. Hence $\int u_+^*d\lambda>0$ and $\util_+$ has nontrivial asymptotic behavior at all its punctures.

We now claim that
\begin{equation}
\label{wind_infty_top_level_end}
\wind_\infty(\util_+,\infty,\tau) \leq 1.
\end{equation}
To prove \eqref{wind_infty_top_level_end}, SFT compactness provides constants $A_n\in\C^*$, $B_n\in\C$ such that the sequence of maps $\wtil_n(z) = \vtil_n(A_nz+B_n)$ has the following properties: 

\begin{itemize}
\item[(i)] For every open neighborhood $U$ of $\Gamma$ in $\C$ and every $A\geq 0$ there exists $n_{U,A}$ such that if $n\geq n_{U,A}$ then $\wtil_n(\C\setminus U) \subset [A,+\infty)\times V^+$.
\item[(ii)] By (i) we can write $\wtil_n=(d_n,w_n)$ on arbitrary compact subsets of $\C\setminus\Gamma$ provided $n$ large enough. Then $w_n \to u_+$ in $C^\infty_{\rm loc}(\C\setminus\Gamma)$.
\end{itemize}

Note that there is a uniformity in neighbourhoods of $\infty$ encoded in property~(i). Write $(s,t)$ instead of $e^{2\pi(s+it)}$. Theorem~\ref{thm_asymptotics} provides a constant $s_0>0$ such that $\pi_\lambda(\partial_su_+)$ does not vanish on $[s_0,+\infty)\times\R/\Z$. In particular the winding number $\wind_\tau(\pi_\lambda(\partial_su_+(s_0,\cdot)))$ of $t\mapsto \pi_\lambda(\partial_su_+(s_0,t))$ in the global frame $\tau$ is equal to $\wind_\infty(\util,\infty,\tau)$. By (i) we can assume, up to making $s_0$ larger, that $\pi_\lambda(\partial_sw_n)$ is defined on the set $[s_0,+\infty)\times\R/\Z$. Use $\tau$ to represent $\pi_\lambda(\partial_sw_n)$ as a $\C$-valued function $\zeta_n(s,t)$. This function satisfies a Cauchy-Riemann type equation, hence its zeros count positively towards the total algebraic count of zeros. Since $\pi_\lambda(\partial_sw_n(s_0,\cdot)) \to \pi_\lambda(\partial_su_+(s_0,\cdot))$ in $C^\infty(\R/\Z)$ and the latter does not vanish, we get that $\zeta_n(s_0,\cdot)$ does not vanish and $\wind(\zeta_n(s_0,\cdot)) = \wind_\infty(\util_+,\infty,\tau)$ for all $n$ large enough. The argument principle now implies that
\[
1\geq \wind_\infty(\vtil_n) = \lim_{s\to+\infty} \wind(\zeta_n(s,\cdot)) \geq \wind(\zeta_n(s_0,\cdot)) = \wind_\infty(\util_+,\infty,\tau)
\]
for all $n\gg1$, as desired.

With \eqref{wind_infty_top_level_end} established, we now invoke the dynamical convexity of $\lambda$ up to action~$C$. Namely, at each negative puncture $z\in\Gamma$ we have $\wind_\infty(\util,z,\tau) \geq 2$ since the corresponding asymptotic limits satisfy $\CZ^0_\tau\geq 3$. Hence 
\[
0 \leq \wind_\pi(\util_+) = \wind_\infty(\util_+) -1 + \#\Gamma \leq 1 - 2\#\Gamma - 1 + \#\Gamma = -\#\Gamma,
\]
so $\Gamma=\emptyset$. This contradiction concludes the proof of the claim that ${\bf u}$ has only one level.

The bottom level of ${\bf u\/}$ cannot contain any closed irreducible components, because the symplectic form $\overline{\omega}$ on $\overline{W}$ is exact. Thus we have shown that a sequence $\vtil_n$ representing a sequence in $\mathcal{Y}_L$ has $C^\infty_{\rm loc}$-convergent subsequence, up to reparametrisations, to a finite-energy plane~$\vtil$ asymptotic to $\gamma_+$.

Denoting again by $\vtil_n$ the convergent subsequence, the SFT compactness provides information analogous to (i) and (ii):
\begin{itemize}
\item[(iii)] There exists some open bounded set $U \subset\C$ and some $N$ such that if $n\geq N$ then $\vtil_n(\C\setminus U) \subset [0,+\infty)\times Y^+$.
\end{itemize}
Now the exact same argument as above will show that $\wind_\infty(\vtil)\leq 1$, i.e. $\vtil$ is fast. Since $\gamma$ is simply covered we know that $\vtil$ is somewhere injective. An isolated self-intersection of~$\vtil$ counts strictly positively to the total self-intersection number, forcing self-intersections of $\vtil_n$ for $n$ large enough, a contradiction. Non-isolated self-intersections would force $\vtil$ to be multiply covered, a contradiction. A direct application of~\cite[Theorem~7.2]{MW} implies that $\vtil$ does not have critical points: critical points of~$\vtil$ are isolated, and~$\vtil$ is a pseudo-holomorphic map that is a $C^\infty_{\rm loc}$-limit of pseudo-holomorphic maps~$\vtil_n$. Hence,~$\vtil$ is an embedding.

In fact,  This is a contradiction. Non-isolated self-intersections would force $\vtil$ to be multiply covered, contradicting the fact that $\gamma_+$ is a simple orbit.

 Similarly, a critical point of $\vtil$ would force that $\vtil_n$ is not an embedding for $n$ large enough, a contradiction. This concludes the proof that $\mathcal{Y}_L$ is compact.

\subsubsection*{Step 4}
We now produce the Reeb orbit $\gamma$ claimed by the Lemma~\ref{lemma_hard_step}. From now on~$\gamma_+$ is assumed to be a Hopf orbit with action $\mathcal{A}_{\rm Hopf}(\Omega_2)$. Existence of $\gamma_+$ follows from the existence of Hopf orbits in $Y^+$ together with an application of Lemma~\ref{lemma_easy_step}. 

The moduli space~$\Mfast(\gamma_+,\jbar)$ is nonempty. To prove this it suffices to show that there exists an embedded fast finite-energy plane in the symplectization $\R \times Y^+$ asymptotic to~$\gamma_+$. This is proved in~\cite[Lemma~3.14]{hryn_jsg}.

The Fredholm theory from Step 2 provides families of planes in $\Mfast(\gamma_+,\jbar)$ by exponentiating nonvanishing sections of the normal bundle. Hence the set of points in $\W$ contained in the image of some plane in $\mathcal{Y}$ is open. 
This fact and the compactness of $\mathcal{Y}_L$ proved above together show that 
we can find a sequence of planes $[\vtil_n] \in\mathcal{Y}$ satisfying
\begin{equation}
\forall L\geq 0 \ \exists n_L \ : \ n\geq n_L \Longrightarrow \vtil_n(\C) \cap \left((-\infty,-L] \times Y^- \right) \neq \emptyset.
\end{equation}
By SFT compactness again, we can pass to a subsequence so that $[\vtil_n]$ converges to a building ${\bf u}$. Unlike in Step 3, ${\bf u}$ now has at least one lower level in $(\R\times Y^-,\jtil)$, in addition to a unique level in $(\W,\jbar)$, and possibly some upper levels in $(\R\times Y^+,\lambda)$. 

The same argument as in Step 3, using (i) and (ii), shows that there are no upper levels in $(\R\times Y^+,\lambda)$.  It follows that the level in $(\W,\jbar)$ consists precisely of a finite energy sphere $\util_0:(\C\setminus\Gamma_0,i) \to (\W,\jbar)$ with precisely one positive puncture (at $\infty$) where the asymptotic limit is $\gamma_+$. If the asymptotic behavior at the positive puncture is trivial then $\wind_\infty(\util_0,\tau,\infty) = -\infty$. If the  asymptotic behavior at $\infty$ is nontrivial trivial then, up to reparametrisation, the SFT compactness theorem provides property (iii), and one can argue similarly to Step 3 to prove that $\wind_\infty(\util_0,\tau,\infty) \leq 1$. 

We claim now that $\util_0$ is an embedded cylinder, asymptotic at its negative puncture to a Reeb orbit $\gamma_-$ satisfying $\CZ(\gamma_-)=3$. To see this, first note that $\util_0$ is somewhere injective because the asymptotic limit at the unique positive puncture $\infty$ is the simple Reeb orbit $\gamma_+$. Since it is a limit of embeddings, the results from \cite{MW} imply that $\util_0$ is also an embedding. Now one builds a weighted Fredholm theory with sections of the normal bundle for embedded curves in $(\W,\jbar)$ with the same asymptotic behavior as $\util_0$. There is only the nontrivial weight $\delta<0$ at the positive end where the asymptotic limit is $\gamma_+$. The weighted Fredholm index is 
\begin{equation*}
\CZ^\delta_\tau(\gamma_+) - \sum_{z\in\Gamma_0} \CZ^0_\tau(\gamma_z) - 1 + \#\Gamma_0
\end{equation*}
Here $\gamma_z$ denotes the asymptotic limit of $\util_0$ at the negative punctures $z\in\Gamma_0$. Standard arguments imply that if $\jbar$ is chosen in a residual subset of almost complex structures, then this weighted Fredholm index is nonnegative. Assume that $\jbar$ has been chosen this way in Step 1. We know that $\CZ^\delta_\tau(\gamma)=3$ and $\CZ^0_\tau(\gamma_z)\geq 3$ by dynamical convexity. Plugging this in to the above formula for the weighted Fredholm index, we obtain
\begin{equation}
0 \leq 3 - 3\#\Gamma_0 -1 + \#\Gamma_0 = 2(1-\#\Gamma_0).
\end{equation}
It follows that $\#\Gamma_0=1$ and the above inequalities are equalities. Thus $\util_0$ is an embedded cylinder asymptotic at its negative puncture to a Reeb orbit $\gamma_-$ satisfying $\CZ(\gamma_-)=3$.

\subsubsection*{Step 5}
We claim now that $\gamma_-$ fulfills the requirements of Lemma~\ref{lemma_hard_step}. We have
\[
\mathcal{A}(\gamma_-) < \mathcal{A}(\gamma_+) \le \mathcal{A}_{\operatorname{Hopf}}(\Omega_2)
\]
by Stokes theorem since $(W,\omega)$ is exact. So we just need to show that $\gamma_-$ is a Hopf orbit.

The argument is similar to before to conclude that there is exactly one lower level which is a fast embedded finite energy plane in $(\R\times Y^-,\jtil)$ asymptotic to $\gamma_-$. Let us provide the details. The level of ${\bf u}$ immediately below the level represented by~$\util_0$ has to be represented by a unique smooth  finite energy sphere $\util_- : (\C\setminus\Gamma_-,i) \to (\R\times Y^-,\jtil)$ with precisely one positive puncture ($\infty$) where it is asymptotic to $\gamma_-$. Note that the Reeb orbit $\gamma_-$ is simple; it follows from dynamical convexity that any iterated orbit satsifies $\CZ^0_\tau \geq 5$. Hence nontriviality of ${\bf u}$ implies that $\int u_-^*d\lambda>0$, i.e. $\util_-$ has nontrivial asymptotic behavior at all its punctures. Then $\CZ^0_\tau(\gamma_-)=3$ implies $\wind_\infty(\util_-,\infty,\tau) \leq 1$. Dynamical convexity implies that $\wind_\infty(\util_-,z,\tau) \geq 2$ at all negative punctures $z\in\Gamma_-$. We can estimate
\begin{equation}
0 \leq \wind_\pi(\util_-) = \wind_\infty(\util_-) - 1 + \#\Gamma_- \leq 1 - 2\#\Gamma_- - 1 + \#\Gamma_- = -\#\Gamma_-
\end{equation}
implying that $\Gamma_-=\emptyset$ and $\util_-$ is a plane. It is automatically fast: $\wind_\infty(\util_-) = \wind_\infty(\util_-,\infty,\tau) \leq 1$. It remains to show that $\util_-$ is embedded, but this follows from the fact that it is a somewhere injective limit of embeddings.

In \cite{fast,hryn_jsg} it is proved\footnote{
In the light of results from \cite{sief_int} this can now be explained in more modern language.  If the embedded fast plane $\util=(a,u)$ in $(\R\times \partial W,\jtil)$ has as asymptotic limit a simply covered periodic orbit $\hat P$ then one can define a weighted self-intersection number $\util *_\delta \util = {\rm int}(\util,\util_\tau) + \alpha^{<\delta}_\tau(\hat P)$. Here $\tau$ is any $d\lambda$-symplectic trivialization of $\xi$ along $\hat P$, and $\util_\tau=(a,u_\tau)$ where $u_\tau$ is the map $u$ pushed in the direction of $\tau$ near the puncture. Arguing as in~\cite{sief_int}, it follows that $\util *_\delta \util\geq 0$ and that if $\util *_\delta \util = 0$ then $u$ is an embedding into $\partial W \setminus \hat P$.
Plugging $\tau=\tau_N$, where~$\tau_N$ is the trivialization that extends over the normal bundle, we get ${\rm int}(\util,\util_{\tau_N})=0$ and $\alpha^{<\delta}_\tau(\hat P) = \alpha^{<\delta}_\tau(\hat P)-1 = 1 - 1 = 0$. Hence $\util *_\delta \util = 0$ and $\hat P$ is unknotted. 
Since $u$ is transverse to the Reeb vector field, the self-linking number of $\hat P$ is $-1$.
}
that a simple Reeb orbit that is the asymptotic limit of an embedded fast finite energy plane is unknotted and has self-linking number $-1$. Thus $\gamma_-$ is a Hopf orbit. Setting $\gamma = \gamma_+$, the proof of Lemma~\ref{lemma_hard_step} is completed.


\section{The first ECH capacity from a Hopf orbit}

We now recall the definition of $c_1^{\rm ECH}$ and prove Theorem~\ref{thm:Hopf_ECH}.

\subsection{Embedded contact homology}
\label{sec:ech}

To begin, let $Y$ be a closed three-manifold and let $\lambda$ be a nondegenerate contact form on $Y$. We now review the definition of embedded contact homology $\ECH(Y,\lambda)$. For simplicity we further assume that $Y$ is a homology sphere, which is sufficient for our purposes; for more details and the case of general three-manifolds see e.g.\ \cite{ind_rev,Hutchings_lectures}.

\begin{definition}
An {\bf ECH generator} is a finite set of pairs $\alpha=\{(\alpha_i,m_i)\}$ where the $\alpha_i$ are distinct simple Reeb orbits, the $m_i$ are positive integers, and $m_i=1$ whenever $\alpha_i$ is hyperbolic\footnote{A nondegenerate Reeb orbit is {\bf hyperbolic\/} if the linearized return map \eqref{eqn:lrm} has real eigenvalues.}. We sometimes denote an ECH generator as a product $\alpha=\prod_i\alpha_i^{m_i}$.
\end{definition}

\begin{definition}
If $\alpha=\{(\alpha_i,m_i)\}$ is an ECH generator, its {\bf symplectic action\/} is
\[
\mathcal{A}(\alpha) = \sum_i m_i \mathcal{A}(\alpha_i).
\]
\end{definition}

\begin{definition}
If $\alpha=\alpha_1^{m_1}\cdots\alpha_k^{m_k}$ is an ECH generator, its {\bf ECH index\/} is the integer $I(\alpha)$ defined by\footnote{To compare with the notation in other papers: $I(\alpha)$ here corresponds to $I(\alpha,\emptyset)$ in \cite[Def.\ 2.15]{ind_rev}, and to $I(\alpha,\emptyset,Z)$ in \cite[Eq.\ (3.4)]{Hutchings_lectures}, where $Z$ is the unique relative homology class of surface in $Y$ with boundary $\alpha$.}
\begin{equation}
\label{eqn:Ialpha}
I(\alpha) = \sum_i\left(m_ic_\tau(\alpha_i) + m_i^2Q_\tau(\alpha_i)\right) + \sum_{i\neq j}m_im_j\operatorname{link}(\alpha_i,\alpha_j) + \sum_i\sum_{l=1}^{m_i}\CZ_\tau(\alpha_i^l).
\end{equation}
Here $\tau$ is any trivialization of $\xi$ over the Reeb orbits $\alpha_i$; the ECH index $I(\alpha)$ does not depend on this choice, although the terms on the right hand side of \eqref{eqn:Ialpha} do. If $\gamma$ is a simple Reeb orbit, then $c_\tau(\gamma)$ denotes the relative first Chern class of $\xi$ with respect to $\tau$ over a surface bounded by $\gamma$, and $Q_\tau(\gamma)$ denotes the relative self-intersection number with respect to $\tau$ of a surface bounded by $\gamma$; see \cite{ind_rev,Hutchings_lectures} for details. Also, $\operatorname{link}$ denotes the linking number, and $\gamma^l$ denotes the Reeb orbit that covers $\gamma$ with multiplicity $l$.
\end{definition}

\begin{remark}
In our situation it is useful to rewrite the index formula \eqref{eqn:Ialpha} as follows. First, it follows from the definitions that if $\gamma$ is a simple Reeb orbit, then for any trivialization $\tau$ of $\xi$ over $\gamma$, the self-linking number of $\gamma$ satisfies
\begin{equation}
\label{eqn:slqc}
\operatorname{sl}(\gamma) = Q_\tau(\gamma) - c_\tau(\gamma).
\end{equation}
Second, if we choose $\tau$ to be a global trivialization of $\xi$ (which we can do, uniquely up to homotopy, since $Y$ is a homology sphere), then $c_\tau(\gamma)=0$ for every Reeb orbit~$\gamma$. Consequently, we can rewrite equation \eqref{eqn:Ialpha} as
\begin{equation}
\label{absolute_grading}
I(\alpha) = \sum_i m_i^2 \sl(\alpha_i) + \sum_{i\neq j} m_im_j\operatorname{link}(\alpha_i,\alpha_j) +  \sum_i \sum_{l=1}^{m_k} \CZ(\alpha_i^l).
\end{equation}
Here $\CZ$ stands for the Conley-Zehnder index taken with respect to~$\tau$.
\end{remark}

\begin{definition}
Define $\operatorname{ECC}(Y,\lambda)$ to be the vector space\footnote{It is also possible to define ECH with $\Z$ coefficients, but we do not need this.} over $\Z/2$ freely generated by the ECH generators, with $\Z$ grading defined by \eqref{absolute_grading}.
\end{definition}

Next, choose a generic $d\lambda$-compatible almost complex structure on $\xi$, and let $\jtil$ denote the corresponding almost complex structure on $\R\times Y$ as in \S\ref{ssec_pseudo_hol_curves}. If $\alpha$ and $\beta$ are ECH generators, define $\mathcal{M}^J(\alpha,\beta)$ to be the moduli space of $\jtil$-holomorphic currents\footnote{A ``holomorphic current'' is a finite linear combination of somewhere injective holomorphic curves with positive integer multiplicities; see \cite[\S3.1]{Hutchings_lectures} for details.} in $\R\times Y$ asymptotic to $\alpha$ as the $\R$ coordinate goes to $+\infty$ and asymptotic to $\beta$ as the $\R$ coordinate goes to $-\infty$. If $J$ is generic, then when $I(\alpha)-I(\beta)=1$, the moduli space $\mathcal{M}^J(\alpha,\beta)/\R$ is finite, where $\R$ acts by translation of the $\R$ coordinate on $\R\times Y$; see \cite[\S5.3]{Hutchings_lectures}.

\begin{definition}
If $J$ is generic as above, we define a differential
\[
\partial_J:\ECC_*(Y,\lambda) \longrightarrow \ECC_{*-1}(Y,\lambda)
\]
as follows. If $\alpha$ is an ECH generator, then
\[
\partial_J\alpha = \sum_{I(\alpha)-I(\beta)=1}\#\frac{\mathcal{M}^J(\alpha,\beta)}{\R} \beta
\]
where $\beta$ is an ECH generator and $\#$ denotes the mod 2 count.
\end{definition}

It is shown in \cite[Thm.\ 7.20]{HT1} that $\partial_J^2=0$. We denote the homology of the chain complex $(\ECC_*(Y,\lambda),\partial_J)$ by $\ECH_*(Y,\lambda,J)$. It was shown by Taubes \cite{taubes_iso} that $\ECH_*(Y,\lambda,J)$ is canonically isomorphic to a version of Seiberg-Witten Floer cohomology as defined by Kronheimer-Mrowka \cite{km}. As explained in \cite[Thm.\ 1.3]{cc2}, this implies that $\ECH_*(Y,\lambda,J)$ in fact depends only on the pair $(Y,\xi)$, so we can denote it by $\ECH_*(Y,\xi)$.

The case of interest for us is where $Y=S^3$ and $\xi$ is the ``standard'' contact structure (the unique tight contact structure) $\xi_0$ arising from realizing $S^3$ as a star-shaped hypersurface in $\R^4$ with the restriction of the standard Liouville form. In this case we have
\begin{equation}
\label{eqn:echs3}
\operatorname{ECH}_*(S^3,\xi_0) = \left\{\begin{array}{cl} \Z/2, & *=0,2,4,\ldots\\ 0, & \mbox{else}.
\end{array}\right.
\end{equation}
See e.g.\ the calculation in \cite[\S3.7]{Hutchings_lectures}.

There is some additional structure on ECH which we will need. First of all, the empty set is a valid ECH generator, with grading $I(\emptyset)=0$. For any $J$ we have $\partial_J\emptyset=0$, because the ECH differential decreases action, as explained in \S\ref{sec:ckech} below.  It follows from the properties of ECH cobordism maps in \cite[Thm.\ 1.9]{cc2} that the homology class $[\emptyset]\in \operatorname{ECH}_0(Y,\lambda,J)$ gives a well-defined invariant
\[
c(\xi) \in \operatorname{ECH}_0(Y,\xi).
\]
In the example \eqref{eqn:echs3}, $c(\xi_0)$ is the generator of $\operatorname{ECH}_0(S^3,\xi_0)$.

There is also a ``$U$ map''
\[
U:\operatorname{ECH}_*(Y,\xi) \longrightarrow \operatorname{ECH}_{*-2}(Y,\xi).
\]
This is induced by a chain map
\[
U_{J,y}:\ECC_*(Y,\lambda) \longrightarrow \ECC_{*-2}(Y,\lambda)
\]
which counts $\jtil$-holomorphic currrents passing through $(0,y)$ where $y\in Y$ is a point not on any Reeb orbit. See~\cite[\S3.8]{Hutchings_lectures} for details about why the $U$ map is well defined. In the example~\eqref{eqn:echs3}, the $U$ map on $\operatorname{ECH}_*(S^3,\xi_0)$ is an isomorphism whenever $*\neq 0$.

\subsection{ECH capacities}
\label{sec:ckech}

Choose a generic almost complex structure $J$ as needed to define the ECH differential $\partial_J$ on $\operatorname{ECC}_*(Y,\lambda)$. It follows from the definition of the almost complex structure $\jtil$ that if $\beta$ is another ECH generator and if $\mathcal{M}^J(\alpha,\beta)\neq\emptyset$, then $\mathcal{A}(\alpha) \ge \mathcal{A}(\beta)$, with equality only if $\alpha=\beta$. In particular, if $\beta$ appears in $\partial_J\alpha$, then
\begin{equation}
\label{eqn:decreaseaction}
\mathcal{A}(\alpha)>\mathcal{A}(\beta).
\end{equation}
It follows that for each $L\in\R$, we have a subcomplex $\operatorname{ECC}_*^L(Y,\lambda)$ of $\operatorname{ECC}_*(Y,\lambda)$, spanned by ECH generators with symplectic action less than $L$.

\begin{definition}
The {\bf filtered ECH\/}, denoted by $\operatorname{ECH}^L(Y,\lambda)$, is the homology of the chain complex $(\operatorname{ECC}_*^L(Y,\lambda),\partial_J)$.
\end{definition}

It is shown in \cite[Thm.\ 1.3]{cc2} that the filtered ECH depends only on $Y,\lambda,L$ and not on the choice of $J$. In addition, inclusion of chain complexes inducs a well-defined chain map
\[
\imath_*^L:\operatorname{ECH}_*^L(Y,\lambda) \longrightarrow \operatorname{ECH}_*(Y,\lambda).
\]

\begin{definition}
\label{def:ckech}
Let $X\subset\R^4$ be a star-shaped domain, let $Y=\partial X$, let $\lambda$ be the contact form on $Y$ given by the restriction of the standard Liouville form, and suppose that $\lambda$ is nondegenerate. If $k$ is a nonnegative integer, define the {\bf $k^{th}$ ECH capacity\/} of $X$ by
\[
c_k(X) = \inf\{L>0 \mid 0\neq \imath_{2k}^L:\operatorname{ECH}_{2k}^L(Y,\lambda)\to \operatorname{ECH}_{2k}(Y,\lambda)\}.
\]
Equivalently, if $J$ is an almost complex structure as needed to define the ECH generator $\partial_J$, then $c_k(X)$ is the minimum $L$ such that the generator of $ECH_{2k}(Y,\lambda,J)$ can be represented by a sum of ECH generators each having action at most $L$.
\end{definition}

The capacity $c_k$ has a unique $C^0$-continuous extension to the space of all star-shaped domains in $\R^4$; see \cite[\S4]{hut_qech} for explanation and for the definition of ECH capacities of more general symplectic manifolds.

\subsection{Proof of Theorem \ref{thm:Hopf_ECH}}
Let $X\subset\R^4$ be a dynamically convex domain. Let $Y=\partial X$ and let $\lambda$ denote the induced contact form on $Y$. By continuity of capacities, we can assume that $\lambda$ is nondegenerate. We now prove that $c_{\operatorname{Hopf}}(X)=c_1^{\operatorname{ECH}}(X)$ in three steps.

\subsubsection*{Step 1} We first prove that $c_{\rm Hopf}(X)\ge c_1^{\ECH}(X)$. One can quickly deduce this step from the result of Edtmair \cite{gss_edtmair} that $c_Z(X) \le \mathcal{A}_{\operatorname{Hopf}}(X)$, but we will give an independent proof.

By Theorem~\ref{main_fast_capacity}, there exists a Hopf orbit $\gamma\in\P_{\operatorname{Hopf}}(\partial X)$ satisfying $c_{\operatorname{Hopf}}(X)=\mathcal{A}(\gamma)$ and $\CZ(\gamma)=3$. It follows from equation \eqref{absolute_grading} that the ECH index $I(\gamma)=2$.

Now choose a generic almost complex structure $J$ as needed to define the ECH differential $\partial_J$, and choose a point $y\in Y$ as needed to define the chain map $U_{J,y}$. We claim that $\partial_J\gamma=0$ and $U_{J,y}\gamma=\emptyset$. Assuming this claim, it follows from the nontriviality of the class $[\emptyset]$ that $\gamma$ is a cycle in the ECH chain complex representing the generator of $\operatorname{ECH}_2(Y,\xi_0)$, so it follows from Definition~\ref{def:ckech} that $c_1^{\operatorname{ECH}}(X)\le \mathcal{A}(\gamma)$, completing the proof of Step 1.

To prepare to prove the claim, recall from Remark~\ref{rem:gss} that since $X$ is dynamically convex and $\gamma$ is a Hopf orbit, it follows that $\gamma$ is the boundary of a disk-like global surface of section $D\subset Y$. In particular, every Reeb orbit other than $\gamma$ has positive linking number with $\gamma$. Moreover, the construction of the global surface of section in \cite{hryn_jsg} obtains $\operatorname{int}(D)$ as the projection to $Y$ of a $J$-holomorphic plane in $\R\times Y$ with a positive puncture at $\gamma$.

To prove that $\partial_J\gamma=0$, we will show that if $\beta$ is an ECH generator with $I(\beta)=1$, then $\mathcal{M}^J(\gamma,\beta)=\emptyset$. Suppose to the contrary that there exists $C\in\mathcal{M}^J(\gamma,\beta)$. We can write $\beta=\beta_1^{m_1}\cdots\beta_k^{m_k}$. Since $I(\beta)\neq 0$, we must have $k>0$. By the inequality \eqref{eqn:decreaseaction}, all of the orbits $\beta_i$ are distinct from $\gamma$.

It follows from the asymptotics in Theorem~\ref{thm_asymptotics} that if $a_0$ is sufficiently large, then $C$ is transverse to $\{\pm a_0\}\times Y$, and
\[
C\cap ([a_0,\infty)\times \gamma)=C\cap((-\infty,a_0]\times \gamma)=\emptyset.
\]
Let $\eta_{\pm}$ denote the projections of $C\cap (\{\pm a_0\}\times Y)$ to $Y$.
We have
\begin{equation}
\label{eqn:lk1}
\link(\eta_{+},\gamma)-\link(\eta_{-},\gamma)=\#(C\cap ([-a_0,a_0]\times\gamma))\ge 0,    
\end{equation}
where $\#$ denotes the algebraic count of intersections, which is nonnegative by intersection positivity of holomorphic curves.

Let $\tau$ be a global trivialization of the contact structure $\xi$, and let $\nu$ denote the asymptotic eigenvalue of $u$ at $\gamma$ as in Theorem~\ref{thm_asymptotics}. Since $\nu<0$ and $\operatorname{CZ}_\tau(\gamma)=3$, it follows from the discussion in \S\ref{sec:CZ} that
\[
\operatorname{wind}_\tau(\nu)\le 1.
\]
On the other hand, it follows from the definition of self-linking number that
\[
\operatorname{link}(\eta_+,\gamma) = \operatorname{wind}_\tau(\nu) + \operatorname{sl}(\gamma).
\]
Since $\operatorname{sl}(\gamma)=-1$, we conclude that
\begin{equation}
\label{eqn:lk2}
\operatorname{link}(\eta_+,\gamma)\le 0.
\end{equation}
On the other hand, we have
\begin{equation}
\label{eqn:lk3}
\operatorname{link}(\eta_-,\gamma) = \sum_{i=1}^km_i\operatorname{link}(\beta_i,\gamma) \ge k \ge 1.
\end{equation}
Combining the inequalities \eqref{eqn:lk1}, \eqref{eqn:lk2}, and \eqref{eqn:lk3} gives $-1\ge 0$, a contradiction.

To prove that $U_{J,y}\gamma=\emptyset$, we can choose $y\in\operatorname{int}(D)$. Let $C\in\mathcal{M}^J(\gamma,\emptyset)$ be a $J$-holomorphic plane that projects to $\operatorname{int}(D)$. By translating $C$ in the $\R$ direction, we can assume that $(0,y)\in C$. We claim now that if $\beta$ is an ECH generator, if $C'\in\mathcal{M}^J(\gamma,\beta)$, and if $C\neq C'$, then $C\cap C'=\emptyset$, and in particular $(0,y)\notin C'$. This will prove that $C$ is the only curve counted by $U_{J,y}\gamma$ so that $U_{J,y}\gamma=\emptyset$.

To prove that $C\cap C'=\emptyset$, observe that by the definition of the relative intersection number in \cite[\S2.7]{ind_rev}, we have
\begin{equation}
\label{eqn:Q1}
Q_\tau(\gamma) = \#(C\cap C') - \ell_\tau(C,C'),
\end{equation}
where $\ell_\tau(C,C')$ is the asymptotic linking number of $C$ and $C'$, which is the linking of the positive ends of $C$ and $C'$ in a neighborhood of $\gamma$ with respect to the trivialization $\tau$; see \cite[Def.\ 2.9]{ind_rev} for the precise definition. Since $\tau$ is a global trivialization and $\operatorname{sl}(\gamma)=-1$, it follows from equation \eqref{eqn:slqc} that
\begin{equation}
\label{eqn:Q2}
Q_\tau(\gamma) = -1.
\end{equation}
On the other hand, it follows from the linking bound in \cite[Lem.\ 4.17]{ind_rev} that
\begin{equation}
\label{eqn:Q3}
\ell_\tau(C,C') \le 1.
\end{equation}
Combining \eqref{eqn:Q1}, \eqref{eqn:Q2}, and \eqref{eqn:Q3}, we obtain $\#(C\cap C') \le 0$. By intersection positivity, it follows that $\#(C\cap C')=0$ and $C\cap C'=\emptyset$ as desired.

\subsubsection*{Step 2} We now show that any cycle in the chain complex $(ECC_*(Y,\lambda),\partial_J)$ that is homologous to $\emptyset$ must include $\emptyset$ as a summand. For this purpose it suffices to show that if $\alpha$ is an ECH generator with $I(\alpha)=1$, then $\mathcal{M}^J(\alpha,\emptyset)=\emptyset$. One can prove this using the Fredholm index, but to prepare for the following step we will give a different proof using the ``$J_0$ index'' defined in \cite[\S6.1]{ind_rev}.

In the present context, if $\alpha=\prod_{i=1}^k\alpha_i^{m_i}$ is an ECH generator, then $J_0(\alpha)$ is an integer defined by
\begin{equation}
\label{eqn:J0def}
J_0(\alpha) = I(\alpha) -  \sum_{i=1}^k \CZ(\alpha_i^{m_i}).
\end{equation}
The key property of $J_0$ that we need is that if $\beta$ is another ECH generator with $I(\alpha)-I(\beta)\in\{1,2\}$, and if there exists $\mathcal{C}\in\mathcal{M}^J(\alpha,\beta)\neq\emptyset$ (i.e. if $\mathcal{C}$ is counted by the ECH differential or the $U$ map from $\alpha$ to $\beta$), then
\begin{equation}
\label{eqn:J0inequality}
J_0(\alpha)-J_0(\beta)\ge -1,
\end{equation}
with equality only if $\mathcal{C}$ consists of a holomorphic plane, possibly together with some trivial cylinders. This is a special case of \cite[Prop.\ 5.8]{Hutchings_lectures}.

Suppose now that $\alpha=\prod_{i=1}^k\alpha_i^{m_i}$ is an ECH generator with $I(\alpha)=1$. Since $I(\alpha)\neq 0$, we must have $k\ge 1$. It then follows from equation \eqref{eqn:J0def} and dynamical convexity that
\[
J_0(\alpha) \le -2.
\]
On the other hand $J_0(\emptyset)=0$. Then if $\mathcal{M}^J(\alpha,\emptyset)\neq \emptyset$, the inequality \eqref{eqn:J0inequality} with $\beta=\emptyset$ gives $-2\ge -1$, a contradiction.

\subsubsection*{Step 3} We now prove that $c_{\rm Hopf}(X)\le c_1^{\ECH}(X)$.

By Definition~\ref{def:ckech}, there is a cycle $\sum_jx_j$ in the ECH chain complex $(\operatorname{ECC}_*(Y,\lambda),\partial_J)$, representing the generator of $\operatorname{ECH}_2(Y,\xi_0)$, such that each summand $x_j$ is an ECH generator with ECH index $I(x_j)=2$ and action $\mathcal{A}(x_j)\le c_1^{\operatorname{ECH}}(X)$. Since $U:\operatorname{ECH}_2(Y,\xi_0)\to \operatorname{ECH}_0(Y,\xi_0)$ is an isomorphism, and since $\emptyset$ is a cycle representing the generator of $\operatorname{ECH}_0(Y,\xi_0)$, it follows that $U_{J,y}\sum_jx_j$ is homologous to $\emptyset$. By Step 2, there is some $j$ such that $\emptyset$ is a summand of $U_{J,y}x_j$, and in particular there exists $\mathcal{C}\in\mathcal{M}^J(x_j,\emptyset)$.

By the inequality \eqref{eqn:J0inequality}, we have $J_0(x_j)\ge -1$. By equation \eqref{eqn:J0def} and dynamical convexity, $x_j$ consists of a single simple Reeb orbit $\gamma$ with multiplicity $1$, and $\operatorname{CZ}(\gamma)=3$. Since $I(\gamma)=2$, it follows from equation \eqref{absolute_grading} that $\operatorname{sl}(\gamma)=-1$. Since equality holds in \eqref{eqn:J0inequality}, $\mathcal{C}$ is a holomorphic plane positively asymptotic to $\gamma$. By \cite[Prop.\ 3.4]{wh}, the projection of $\mathcal{C}$ to $Y$ is an embedding\footnote{The statement of \cite[Prop.\ 3.4]{wh} requires that $\gamma$ is elliptic, but the proof works just as well if $\gamma$ is negative hyperbolic, so it applies given that $\operatorname{CZ}(\gamma)=3$. The basic ideas in the proof go back to \cite{props2}, and a more general statement is proved in \cite[Theorem~2.6]{sief_int}. One can further show that the projection of $\mathcal{C}$ to $Y$ is the interior of a global surface of section for the Reeb flow, but we do not need this.}. Thus $\gamma$ is unknotted in $Y$, so $\gamma$ is a Hopf orbit. Then $c_{\operatorname{Hopf}}(X)\le \mathcal{A}(\gamma)$, and we know from the previous paragraph that $\mathcal{A}(\gamma) \le c_1^{\operatorname{ECH}}(X)$.



\begin{thebibliography}{99}

\bibitem{ABE} A. Abbondandolo, G. Benedetti and O. Edtmair, {\em Symplectic capacities of domains close to the ball and Banach-Mazur geodesics in the space of contact forms\/}, arXiv:2312.07363.

\bibitem{ABHS_2sphere} A. Abbondandolo, B. Bramham, U. Hryniewicz and P. A. S. Salom\~ao, \textit{A systolic inequality for geodesic flows on the two-sphere}, Math. Ann. 367 (2017), no. 1-2, 701--753.

\bibitem{ABHS_3sphere} A. Abbondandolo, B. Bramham, U. Hryniewicz and P. A. S. Salom\~ao, \textit{Sharp systolic inequalities for Reeb flows on the three-sphere}, Invent. Math. 211 (2018), no. 2, 687--778.

\bib{ABHS_comp}{article} {
    AUTHOR = {Abbondandolo, A.},
    AUTHOR = {Bramham, B.},
    AUTHOR = {Hryniewicz, U.},
    AUTHOR = {Salom\~{a}o, P. A. S.},
     TITLE = {Systolic ratio, index of closed orbits and convexity for tight
              contact forms on the three-sphere},
   JOURNAL = {Compos. Math.},
    VOLUME = {154},
      YEAR = {2018},
    NUMBER = {12},
     PAGES = {2643--2680},
      ISSN = {0010-437X},
}

\bibitem{aek} A. Abbondandolo, O. Edtmair and J. Kang, \textit{On closed characteristics of minimal action on a convex three-sphere}, \texttt{arXiv:2412.01777}

\bibitem{ako} S. Arstein-Avidan, R. Karasev, and Y. Ostrover, {\em From symplectic measurements to the Mahler conjecture\/}, Duke Math. J. {\bf 163} (2014), 2003--2022.

\bib{sft_comp}{article} {
    AUTHOR = {Bourgeois, F.},
    AUTHOR = {Eliashberg, Y.},
	AUTHOR = {Hofer, H.},
	AUTHOR = {Wysocki, K.},
	AUTHOR = {Zehnder, E.},
     TITLE = {Compactness results in symplectic field theory},
   JOURNAL = {Geom. Topol.},
    VOLUME = {7},
      YEAR = {2003},
     PAGES = {799--888},
      ISSN = {1465-3060},
}

\bibitem{chaidezedtmair}
J. Chaidez and O. Edtmair, {\em 3D convex contact forms and the Ruelle invariant\/}, Invent. Math. {\bf 229} (2022), 243--301.

\bibitem{chaidezhutchings} J. Chaidez and M. Hutchings, {\em Computing Reeb dynamics on four-dimensional convex polytopes\/}, Journal of Computational Dynamics {\bf 8} (2021), 403--445.

\bibitem{chls} K. Cieliebak, H. Hofer, J. Latschev, and F. Schlenk, {\em Quantitative symplectic geometry\/}, in Dynamics, Ergodic Theory, and Geometry, B. Hasselblatt, Ed., Cambridge University Press, 2007, pp. 1--44.

\bibitem{c-gh} D. Cristofaro-Gardiner and R. Hind, {\em On the large-scale geometry of domains in an exact symplectic 4-manifold\/}, arXiv:2311.06421.

\bibitem{dgrz} J. Dardennes, J, Gutt, V. Ramos, and J. Zhang, {\em Coarse distance from dynamically convex to convex\/}, arXiv:2308.06604.

\bibitem{gss_edtmair} O. Edtmair, \textit{Disk-Like Surfaces of Section and Symplectic Capacities}, Geom. Funct. Anal. 34, 1399--1459 (2024).

\bibitem{ekelandhofer} I. Ekeland and H. Hofer, {\em Symplectic topology and Hamiltonian dynamics\/}, Math. Z. {\bf 200} (1989), 355--378.

\bib{gromov}{article}{
    AUTHOR = {Gromov, M.},
     TITLE = {Pseudo holomorphic curves in symplectic manifolds},
   JOURNAL = {Invent. Math.},
    VOLUME = {82},
      YEAR = {1985},
    NUMBER = {2},
     PAGES = {307--347},
      ISSN = {0020-9910,1432-1297},
}

\bibitem{gutthutchings}
J. Gutt and M. Hutchings, {\em Symplectic capacities from positive $S^1$-equivariant symplectic homology\/}, Algebr. Geom. Topol. {\bf 18} (2018), 3537--3600.


\bib{GHR}{article} {
	AUTHOR = {Gutt, J.},
 Author = {Hutchings, M.},
 Author = {Ramos, V. G. B.},
     TITLE = {Examples around the strong {V}iterbo conjecture},
   JOURNAL = {J. Fixed Point Theory Appl.},
    VOLUME = {24},
      YEAR = {2022},
    NUMBER = {2},
     PAGES = {Paper No. 41, 22},
      ISSN = {1661-7738,1661-7746},
}

\bibitem{guttramos} J. Gutt and V. G. B. Ramos, \textit{The equivalence of Ekeland-Hofer and equivariant symplectic homology capacities}, \texttt{arXiv:2412.09555}

\bibitem{pazit} P. Haim-Kislev, {\em The EHZ capacity of cubes and simplices\/}, Tel Aviv University masters thesis, 2019.

\bibitem{HKO} P. Haim-Kislev and Y. Ostrover, \textit{A Counterexample to Viterbo's Conjecture}, \texttt{arXiv:2405.16513}


\bib{93}{article} {
    AUTHOR = {Hofer, H.},
     TITLE = {Pseudoholomorphic curves in symplectizations with applications to the {W}einstein conjecture in dimension three},
   JOURNAL = {Invent. Math.},
    VOLUME = {114},
      YEAR = {1993},
    NUMBER = {3},
     PAGES = {515--563},
      ISSN = {0020-9910},
}

\bib{props1}{article} {
    AUTHOR = {Hofer, H.},
    AUTHOR = {Wysocki, K.},
    AUTHOR = {Zehnder, E.},
     TITLE = {Properties of pseudoholomorphic curves in symplectisations. {I}. {A}symptotics},
   JOURNAL = {Ann. Inst. H. Poincar\'{e} Anal. Non Lin\'{e}aire},
    VOLUME = {13},
      YEAR = {1996},
    NUMBER = {3},
     PAGES = {337--379},
      ISSN = {0294-1449}
}

\bib{props2}{article} {
    AUTHOR = {Hofer, H.},
    AUTHOR = {Wysocki, K.},
    AUTHOR = {Zehnder, E.},
     TITLE = {Properties of pseudo-holomorphic curves in symplectisations {II}: {E}mbedding controls and algebraic invariants},
   JOURNAL = {Geom. Funct. Anal.},
    VOLUME = {5},
      YEAR = {1995},
    NUMBER = {2},
     PAGES = {270--328},
      ISSN = {1016-443X},
}

\bib{props3}{article} {
    AUTHOR = {Hofer, H.},
    AUTHOR = {Wysocki, K.},
    AUTHOR = {Zehnder, E.},
     TITLE = {Properties of pseudoholomorphic curves in symplectizations. {III}. {F}redholm theory},
 BOOKTITLE = {Topics in nonlinear analysis},
    SERIES = {Progr. Nonlinear Differential Equations Appl.},
    VOLUME = {35},
     PAGES = {381--475},
 PUBLISHER = {Birkh\"{a}user, Basel},
      YEAR = {1999},
}

\bib{convex}{article} {
    AUTHOR = {Hofer, H.},
    AUTHOR = {Wysocki, K.},
    AUTHOR = {Zehnder, E.},
     TITLE = {The dynamics on three-dimensional strictly convex energy surfaces},
   JOURNAL = {Ann. of Math. (2)},
    VOLUME = {148},
      YEAR = {1998},
    NUMBER = {1},
     PAGES = {197--289}
}

\bib{small_area}{article} {
    AUTHOR = {Hofer, H.},
    AUTHOR = {Wysocki, K.},
    AUTHOR = {Zehnder, E.},
     TITLE = {Finite energy cylinders of small area},
   JOURNAL = {Ergodic Theory Dynam. Systems},
    VOLUME = {22},
      YEAR = {2002},
    NUMBER = {5},
     PAGES = {1451--1486},
      ISSN = {0143-3857}
}

\bibitem{hry_thesis} U. Hryniewicz, {\it Finite energy foliations of convex sets in $\mathbb{R}^4$} (PhD thesis), ProQuest LLC, Ann Arbor, MI, 2008

\bib{fast}{article} {
    AUTHOR = {Hryniewicz, U.},
     TITLE = {Fast finite-energy planes in symplectizations and
              applications},
   JOURNAL = {Trans. Amer. Math. Soc.},
    VOLUME = {364},
      YEAR = {2012},
    NUMBER = {4},
     PAGES = {1859--1931},
      ISSN = {0002-9947},
}

\bib{hryn_jsg}{article} {
    AUTHOR = {Hryniewicz, U.},
     TITLE = {Systems of global surfaces of section for dynamically convex {R}eeb flows on the 3-sphere},
   JOURNAL = {J. Symplectic Geom.},
    VOLUME = {12},
      YEAR = {2014},
    NUMBER = {4},
     PAGES = {791--862},
      ISSN = {1527-5256},
}

\bib{ind_rev}{article} {,
    AUTHOR = {Hutchings, Michael},
     TITLE = {The embedded contact homology index revisited},
 BOOKTITLE = {New perspectives and challenges in symplectic field theory},
    SERIES = {CRM Proc. Lecture Notes},
    VOLUME = {49},
     PAGES = {263--297},
 PUBLISHER = {Amer. Math. Soc., Providence, RI},
      YEAR = {2009},
}

\bib{hut_qech}{article}{
    AUTHOR = {Hutchings, Michael},
     TITLE = {Quantitative embedded contact homology},
   JOURNAL = {J. Differential Geom.},
    VOLUME = {88},
      YEAR = {2011},
    NUMBER = {2},
     PAGES = {231--266},
      ISSN = {0022-040X,1945-743X},
}

\bibitem{Hutchings_lectures}
M. Hutchings, {\em Lecture notes on embedded contact homology\/}, in Contact and Symplectic Topology, Bolyai Society Mathematical Studies, Springer, 2014, pp. 389--484.

\bibitem{Hutchings_Ruelle}
M. Hutchings, {\em ECH capacities and the Ruelle invariant\/}, J. Fixed Point Theory and Applications {\bf 24} (2022), Paper No. 50, 25 pp.

\bibitem{altech} M. Hutchings, {\em An elementary alternative to ECH capacities\/}, PNAS vol.\ {\bf 119}, No.\ 35, e2203090119 (2022).

\bib {HT1}{article}{
    AUTHOR = {Hutchings, Michael},
    Author = {Taubes, Clifford Henry},
     TITLE = {Gluing pseudoholomorphic curves along branched covered
              cylinders. {I}},
   JOURNAL = {J. Symplectic Geom.},
    VOLUME = {5},
      YEAR = {2007},
    NUMBER = {1},
     PAGES = {43--137},
      ISSN = {1527-5256,1540-2347},
       URL = {http://projecteuclid.org/euclid.jsg/1197491304},
}

\bibitem{wh} M. Hutchings and C. H. Taubes, {\em The Weinstein conjecture for stable Hamiltonian structures\/}, Geom. Topol. {\bf 13} (2009), 901--941.

\bibitem{cc2} M. Hutchings and C. H. Taubes, {\em Proof of the Arnold chord conjecture in three dimensions, II\/}, Geom. Topol. {\bf 17} (2013), 2601--2688.

\bibitem{km} P. B. Kronheimer and T. S. Mrowka, {\em Monopoles and three-manifolds\/}, Cambridge University Press, 2008.

\bib{MW}{article} {
    AUTHOR = {Micallef, M. J.},
    AUTHOR = {White, B.},
     TITLE = {The structure of branch points in minimal surfaces and in pseudoholomorphic curves},
   JOURNAL = {Ann. of Math. (2)},
    VOLUME = {141},
      YEAR = {1995},
    NUMBER = {1},
     PAGES = {35--85},
       URL = {https://doi.org/10.2307/2118627},
}

\bibitem{rabinowitz}
P. Rabinowitz, {\em Periodic solutions of Hamiltonian systems\/}, Comm. Pure. Appl. Math {\bf 31} (1978), 157--184.

\bibitem{schlenk}
F. Schlenk, {\em Symplectic embedding problems, old and new\/}, Bull. AMS {\bf 55} (2018), 139--182.

\bibitem{taisuke1}
T. Shibata, {\em Dynamically convex and global surface of section in $L(p,p-1)$ from the viewpoint of ECH\/}, arXiv:2306.04132.

\bibitem{taisuke2}
T. Shibata, {\em Elliptic bindings and the first ECH spectrum for convex Reeb flows on lens spaces\/}, arXiv:2309.09133.

\bib{sie_CPAM}{article} {
    AUTHOR = {Siefring, Richard},
     TITLE = {Relative asymptotic behavior of pseudoholomorphic half-cylinders},
   JOURNAL = {Comm. Pure Appl. Math.},

    VOLUME = {61},
      YEAR = {2008},
    NUMBER = {12},
     PAGES = {1631--1684},
      ISSN = {0010-3640},
}

\bib{sief_int}{article} {
    AUTHOR = {Siefring, Richard},
     TITLE = {Intersection theory of punctured pseudoholomorphic curves},
   JOURNAL = {Geom. Topol.},
    VOLUME = {15},
      YEAR = {2011},
    NUMBER = {4},
     PAGES = {2351--2457},
      ISSN = {1465-3060},
}


\bib {taubes_iso}{article}{
    AUTHOR = {Taubes, Clifford Henry},
     TITLE = {Embedded contact homology and {S}eiberg-{W}itten {F}loer
              cohomology {I}},
   JOURNAL = {Geom. Topol.},
    VOLUME = {14},
      YEAR = {2010},
    NUMBER = {5},
     PAGES = {2497--2581},
      ISSN = {1465-3060,1364-0380},
}

\bib{viterbo}{article}{
    AUTHOR = {Viterbo, C.},
     TITLE = {Metric and isoperimetric problems in symplectic geometry},
   JOURNAL = {J. Amer. Math. Soc.},
    VOLUME = {13},
      YEAR = {2000},
    NUMBER = {2},
     PAGES = {411--431},
      ISSN = {0894-0347},
}


\end{thebibliography}
\end{document}